\documentclass[microtype]{gtpart}
\agtart
\usepackage{sseq,pgffor}
\usepackage{todonotes}
\usepackage[a4paper]{geometry}
\usepackage{array,enumerate,graphicx,etoolbox,tabularx,booktabs}
\usepackage[all]{xy}
\usepackage{mathtools}
\usepackage{xcolor} %improved colors
\usepackage[nameinlink,capitalise,noabbrev]{cleveref}
\usepackage{tikz-cd}
\usepackage{stackrel}
\usetikzlibrary{decorations.pathmorphing}
\usepackage[shortlabels]{enumitem} % for '\newlist' and '\setlist' macros
\DeclareMathOperator{\Mod}{Mod}

\newcount\tmp
\newcommand{\F}{\mathbb{F}}

\newcommand{\G}{\mathbb{G}}
\newcommand{\BP}{\mathbf{BP}}
\newcommand{\E}{\mathbf{E}}
\newcommand{\K}{\mathbf{K}}
\renewcommand{\k}{\mathbf{k}}

\newcommand{\hhr}{\mathbf{HHR}}
\newcommand{\N}{\mathbb{N}}

\newcommand{\veff}{\mathbf{veff}}
\newcommand{\eff}{\mathbf{eff}}
\newcommand{\xr}{\xrightarrow}

\renewcommand{\cal}{\mathcal}
\newcommand{\cD}{\mathcal{D}}

\usepackage{comment}

\let\Re\relax
\DeclareMathOperator{\Re}{Re} 
\renewcommand{\theequation}{\thesection.\alph{equation}}

\makeatletter
\newcommand*{\myfnsymbolsingle}[1]{%
  \ensuremath{%
    \ifcase#1% 0
    \or % 1
      *%   
    \or % 2
      \dagger
    \or % 3  
      \ddagger
    \or % 4   
      \mathsection
    \or % 5
      \mathparagraph
    \else % >= 6
      \@ctrerr  
    \fi
  }%   
}   
\makeatother

\usepackage{alphalph}
\newalphalph{\myfnsymbolmult}[mult]{\myfnsymbolsingle}{}
\newcommand{\M}{\mathbf{M}}

\DeclareMathOperator{\Spec}{Spec}

\DeclareMathOperator{\Hom}{Hom}

\DeclareMathOperator{\Ext}{Ext}

\DeclareMathOperator{\MU}{MU}

\DeclareMathOperator*{\colim}{colim}

\newcommand{\reg}{\mathbf{reg}}

\newcommand{\cK}{\cal{K}}

\newtheorem{thm}{Theorem}[section]

\newtheorem{lem}[thm]{Lemma} 
\newtheorem{cor}[thm]{Corollary}

\newtheorem{prop}[thm]{Proposition} \theoremstyle{definition}
\newtheorem{rem}[thm]{Remark} 
\newtheorem{defn}[thm]{Definition}
\theoremstyle{definition}
\newtheorem{exmp}[thm]{Example}
 
\newcommand{\unit}{\mathbf{1}}

\Crefname{cor}{Corollary}{Corollaries}
\Crefname{conjecture}{Conjecture}{Conjectures}
\Crefname{rem}{Remark}{Remarks}
\Crefname{prop}{Proposition}{Propositions}	
\Crefname{question}{Question}{Questions}
\Crefname{figure}{Figure}{Figures}
\Crefname{lem}{Lemma}{Lemmas}
\Crefname{thm}{Theorem}{Theorems}

\newcommand{\cC}{\mathcal{C}}

\DeclareMathOperator{\SH}{\mathbf{SH}}  
\DeclareMathOperator{\Sm}{\mathbf{Sm}}

\DeclareMathOperator{\cell}{\mathbf{cell}}
\DeclareMathOperator{\KQ}{\mathbf{KQ}}
\DeclareMathOperator{\KGL}{\mathbf{KGL}}
\DeclareMathOperator{\kgl}{\mathbf{kgl}}
\DeclareMathOperator{\MGL}{\mathbf{MGL}}
\DeclareMathOperator{\Cell}{Cell}
\DeclareMathOperator{\CAlg}{CAlg}

\title{On equivariant and motivic slices}
  \date{\today}
  \author{Drew Heard}
\givenname{Drew}
\surname{Heard}
\address{Fakult{\"a}t f{\"u}r Mathematik, Universit{\"a}t Regensburg}
  \email{drew.k.heard@gmail.com}
\urladdr{https://drew-heard.github.io/}

\arxivreference{1807.09092}

\begin{document}
\begin{abstract}
  Let $k$ be a field with a real embedding. We compare the motivic slice filtration of a motivic spectrum over $\Spec(k)$ with the $C_2$-equivariant slice filtration of its equivariant Betti realization, giving conditions under which realization induces an equivalence between the associated slice towers. In particular, we show that, up to reindexing, the towers agree for all spectra obtained from localized quotients of $\MGL$ and $M\R$, and for motivic Landweber exact spectra and their realizations. As a consequence, we deduce that equivariant spectra obtained from localized quotients of $M\R$ are even in the sense of Hill--Meier, and give a computation of the slice spectral sequence converging to $\pi_{*,*}\BP\langle n \rangle/2$ for $1 \le n \le \infty$. 
\end{abstract}
\maketitle

%\tableofcontents
\section{Introduction}
The slice filtration in equivariant or motivic homotopy theory is an analog of the Postnikov filtration in ordinary homotopy theory that takes into account the more complicated structure of these categories. In the equivariant setting, the slice filtration was used by Hill, Hopkins, and Ravenel \cite{hhr} to give a stunning solution of the Kervaire invariant one problem, while in motivic homotopy theory it has been used to construct a version of the Atiyah--Hirzenbruch spectral sequence for $\KGL$, the motivic spectrum representing algebraic $K$-theory \cite{levine_tower}. More recently, the motivic slice spectral sequence has been used by R{\"o}ndigs, Spitzweck, and {\O}stv{\ae}r to compute the first Milnor--Witt stem over of the motivic sphere spectrum over a general field \cite{rso_slices}. These results indicate the fundamental importance of the slice filtration in both motivic and equivariant homotopy theory. 

As is well known, there are many similarities between the stable motivic homotopy category over $\Spec(\R)$ and the $C_2$-equivariant motivic category; indeed, there is a Betti realization functor $\Re$ from the former to the latter, and one can ask about the relation between the slice filtration for a motivic spectrum $E$, and the corresponding $C_2$-equivariant slice filtration of $\Re(E)$. A priori there is no reason that these should be related, as they are evidently defined differently - for example, the motivic slice functors are triangulated, while the equivariant ones are not. On the other hand, Betti realization takes the motivic spectrum $\KGL$ representing algebraic $K$-theory to Atiyah's real $K$-theory spectrum $K\mathbb{R}$, and the non-zero motivic and equivariant slices are given respectively by
\[
s_q(\KGL) \simeq \Sigma^{2q,q}\M\Z \quad \text{ and } \quad P^{2q}_{2q} K\mathbb{R} \simeq 
  \Sigma^{2q,q\sigma}H\underline{\Z} 
\]
for $q \in \Z$. This notation will be explained more in \Cref{sec:eff} and \Cref{sec:equivariant} respectively, however the main point to note is that Betti realization takes $\Sigma^{2q,q}\M\Z$ to $\Sigma^{2q,q\sigma}H\underline{\Z}$, so that Betti realization takes the $q$-th motivic slice of $\KGL$ to the $2q$-th slice of $K\mathbb{R}$. 

Our main result implies the stronger statement that, up to re-indexing, Betti realization induces an equivalence between the slice \emph{towers} of $\KGL$ and $K\R$. More generally, in \Cref{thm:effvshhr} we give precise conditions on when Betti realization is compatible with the slices towers of a motivic spectrum $E$ and its realization $\Re(E)$. These conditions are satisfied when $E = \MGL$, the motivic spectrum representing algebraic cobordism, and its $p$-local variant $\MGL_{(p)}$, or more generally by quotients and localizations of these by elements coming from the Lazard ring via the natural morphism $L \to \MGL$. We call these motivic spectra localized quotients of $\MGL$ (see \Cref{def:chromatictype} for a precise definition). This definition also makes sense in the $C_2$-equivariant context, where the role of $\MGL$ is played by real cobordism $M\R$. A specialization of \Cref{thm:effvshhr} is then the following. 
\begin{thm}
 Let $k \subseteq \R$ be a field,  $E^{\text{mot}} \in \SH(k)$ a localized quotient of $\MGL$, and $E^{\text{equiv}}$ its $C_2$-equivariant Betti realization. Then, for all $q \in \Z$, there are equivalences 
  \[
\Re(s_q(E^{\text{mot}})) \xr{\simeq} P_{2q}^{2q}(E^{\text{equiv}}),
  \]
  and the odd slices of $E^{\text{equiv}}$ vanish. 
\end{thm}
As with the case of $\KGL$ above (which is in fact a special case), we in fact prove that Betti realizaton induces an equivalence between the slice towers of $E^{\text{mot}}$ and $E^{\text{equiv}}$ - for a precise statement, see \Cref{thm:comparasionthmchromatic}.  Using work of Hill--Hopkins--Ravenel we deduce the following corollary on the bigraded homotopy groups of $E^{\text{equiv}}$. 
\begin{cor}
  Let $E^{\text{equiv}}$ be a localized quotient of $M\R$, then
  \[
\underline{\pi}_{2k-1,k}E^{\text{equiv}} = 0
  \]
  for all $k \in \Z$. 
\end{cor}
Such a result was previously known for quotients of $BP\R$ by work of Greenlees--Meier \cite[Corollary 4.6]{greenlees_meier}, however they proceed by direct computation, whereas our result is a consequence of the computation of the slices of $E^{\text{equiv}}$. 

The comparison theorem gives a morphism of exact couples between the motivic and equivariant slice spectral sequences, and hence a morphism of spectral sequences. Now suppose that $k$ is a real closed field. Using that there is an injection $\pi_{*,*}\M\F_2 \to \pi_{*,*}H\underline{\F}_2$ from the motivic cohomology of a point to the equivariant cohomology of a point, one can leverage the existing equivariant computations to give computations in motivic homotopy theory. We take this up in \Cref{sec:mod2calcs} where, as an example, we compute $\pi_{*,*}\BP\langle n \rangle/2$ for $0 \le n \le \infty$, at least up to extension. The case $n = \infty$ was previously known by work of Yagita \cite{yagita}. When $n = 1$ we recover, up to a non-trivial extension problem, the computation of the mod 2 algebraic $K$-theory of $\R$ by Suslin. 
\subsection*{Outline of the proof}
The proof of our main theorem works by first passing through various different slices filtrations on the motivic stable homotopy category. Many motivic spectra of interest (including localized quotients of $\MGL$) are \emph{cellular}, meaning that they can be constructed out of colimits and extensions of the bigraded motivic spheres $S^{a,b}$ for $a,b \in \Z$. A consequence of the computations of the slices of the sphere spectrum is that if a motivic spectrum is cellular, then so are its slices. In \Cref{sec:cellular} we construct the analog of Voevodsky's slice filtration on the cellular motivic category, and show that for any cellular spectrum, the slices and cellular slices agree, see \Cref{thm:effcellcompare}. 

There is an alternative to Voevodsky's slice filtration known as the very effective slice filtration \cite{mot_twisted,1610.01346}. In \Cref{sec:veff} we introduce the cellular version of this slice filtration, and give conditions on when this agrees with the effective slice filtration. In particular, we show in \Cref{thm:effvsveff,lem:prople} that this is true when $E$ is a localized quotient of $\MGL$ or is Landweber exact. 

The cellular very effective slice filtration in motivic homotopy is similar to the Hill--Hopkins--Ravenel slice filtration in $C_2$-equivariant homotopy theory. Using an abstract version of a theorem of Pelaez, proved here in \Cref{sec:pelaez}, we study the relationship between the very effective cellular slice filtration of a cellular motivic spectrum $E$, and the Hill--Hopkins--Ravenel slice filtration of its realization, cumulating in \Cref{thm:veffcellc2compthm}. Combining this with the results in the previous paragraphs gives our main theorem. 
\subsection*{Conventions}
Throughout we work with $\infty$-categories, specifically the quasi-categories of Joyal and Lurie \cite{lurie_htt,ha} - the results could equally well be proved using the theory of stable model categories, similar to \cite{grso_slices}. We will always use the terminology limit and colimit for homotopy limit and homotopy colimit. 

A full subcategory $\cal{U}$ of a presentable stable $\infty$-category  $\cC$ is called thick if it is a stable subcategory closed under retracts. It is called localizing if it is thick and is additionally closed under arbitrary colimits.  
\subsection*{Acknowledgments}
We thank Lennart Meier for helpful discussions on $C_2$-equivariant homotopy theory, and Tom Bachmann for his help with some of the arguments in \Cref{sec:veff}. We are also grateful to Oliver R{\"o}ndigs and Markus Spitzweck for discussions during a visit to the University of Osnabr\"uck. This work was started while the author was at Hamburg University, supported by DFG-Schwerpunktprogramm 1786. We thank both Hamburg University and Haifa University for their hospitality.  Finally, we thank the referee for many useful suggestions and corrections. 
\section{Slices of stably monoidal categories}
In this section we define the notion of a slice filtration on a symmetric monoidal stable $\infty$-category $\cC$, and give a version of a theorem of Pelaez on the functoriality of the slices with respect to an exact functor $F\colon \cC \to \cD$ between categories equipped with slice filtrations.  
\subsection{An axiomatic approach to the slice filtration}
Let $\cC$ be a symmetric monoidal stable presentable $\infty$-category, compactly generated by a set $\cal{G}$ of objects, and such the tensor product commutes with colimits in both variables. We will call such a category a stably monoidal category. Following \cite[Section 2]{grso_slices} we begin by axiomatizing the notion of a slice filtration. 
\begin{defn}\label{defn:slicefiltration}
  Let $\cC$ be as above. Let $\{\cC_i \}_{i \in \Z}$ be a family of full subcategories of $\cC$. We say that $\{ \cC_i \}_{i \in \Z}$ is a slice filtration of $\cC$ if the $\cC_i$ satisfy the following conditions. 
  \begin{enumerate}[label=(A\arabic*)]
  \item Each $\cC_i$ is closed under equivalences. \label{a1}
  \item $\cC_{i+1} \subseteq \cC_i$ for all $i \in \Z$. \label{a2}
  \item Each $\cC_i$ is generated under colimits and extensions by a set of compact objects $\cK_i$.\label{a3}
  \item The tensor unit is in $\cC_0$. \label{a4}
  \item Each $g \in \cal{G}$ is contained in some $\cC_i$. \label{a5}
  \item If $X \in \cC_0$ and $Y \in \cC_n$, then $X \otimes Y$ is in $\cC_n$ \label{a6}
\end{enumerate}

\end{defn}
\begin{rem}
  Note that our axioms are slightly different from those definition given in \cite{grso_slices}, and in particular we do not require the stronger statement that if $X \in \cC_n$ and $Y \in \cC_m$, then $X \otimes Y \in \cC_{n+m}$. Our motivation comes from equivariant homotopy theory, where the slice filtration of Hill, Hopkins, and Ravenel \cite{hhr} satisfies \ref{a6}, but does not satisfy the stronger multiplicative relation, see \cite[Proposition 2.23]{hill_primer}. We note that the regular slice filtration introduced by Ullman does satisfy this stronger condition, see \cite[Proposition 4.2]{ullman_thesis}.
\end{rem}
As in \cite{grso_slices}, these subcategories are not necessarily closed under desuspension and fibers.  Since $\cC$ is assumed to be presentable, each $\cC_i$ is also a presentable $\infty$-category by \ref{a3} and \cite[Proposition 1.4.4.11]{ha}. 

The following is well-known to hold at the level of triangulated categories, or model categories - we provide a proof for completeness. 
\begin{lem}\label{lem:colim}
  For each $q \in \Z$ the inclusion $i_q \colon \cC_q \hookrightarrow \cC$ has a right adjoint $r_q$ that commutes with filtered colimits. 
\end{lem}
\begin{proof}
  Since the inclusion functor preserves colimits, it has a right adjoint by the adjoint functor theorem \cite[Corollary 5.5.2.9]{lurie_htt}. Note that, by construction, the inclusion functor preserves compact objects. If follows from \cite[Proposition 5.5.7.2]{lurie_htt} that $r_q$ preserves filtered colimits. 
\end{proof}
\begin{rem}
  If $\cC_q$ is a stable subcategory, then $r_q$ commutes with all colimits \cite[Proposition 1.4.4.1(2)]{ha}.
\end{rem}
\begin{defn}
  For any $E \in \cC$, we define $f_q(E) = i_q \circ r_q(E)$. 
\end{defn}
The following is a standard example. 
\begin{exmp}\label{exmp:postnikov}
    Taking $\cC = (\operatorname{Sp},\mathbb{S},\otimes)$ to the category of spectra and $\cK_q = \{\Sigma^m \mathbb{S} \mid m \ge q \}$, we see that $f_q(E)$ is the $(q-1)$-connected cover of $E$, and the diagram
  \[
\cdots \to f_{q+1}E \to f_qE \to f_{q-1}E \to\cdots
  \]
is the dual  Postnikov tower of $E$. 
\end{exmp}
The counit of the adjunction gives rise to a morphism $f_qE \to E$. In fact, by \cite[Proposition 1.4.4.11]{ha} there exists a $t$-structure on $\cC$ such that $\cC_q = \cC_{\ge 0}$. In particular, there is a functorial cofiber sequence
\[
f_qE \to E \to c_qE
\]
such that $f_qE \in \cC_q$ and $c_qE \in \cC_q^{\perp}$, the full subcategory of $\cC$ consisting of those $X \in \cC$ such that $\Hom_{\cC}(Y,X)$ is contractible for all $Y \in \cC_q$.  The map $f_qE \to E$ is characterized up to a contractible space of choices by the properties that $f_qE \in \cC_q$ and $\Hom_{\cC}(M,f_qE) \xr{\simeq} \Hom_{\cC}(M,E)$ for all $M \in \cC_q$. Similarly, the map $E \to c_qE$ is characterized by the properties that $c_qE \in \cC_q^{\perp}$ and $\Hom_{\cC}(c_qE,N) \xr{\simeq} \Hom_{\cC}(E,N)$ for all $N \in \cC_q^{\perp}$. 
This leads to the following recognition principle which is proved as in \cite[Lemma 4.16]{hhr}.
\begin{lem}\label{lem:recognitionconn}
  Suppose there is a fiber sequence 
  \[
F_q \to E \to C_q
  \]
  such that $F_q \in \cC_q$ and $C_q \in (\cC_q)^{\perp}$. Then, the canonical maps $F_q \to f_qE$ and $c_qE \to C_q$ are equivalences. 
\end{lem}
Since $\cC_{q+1} \subseteq \cC_{q}$ it is not hard to verify that the natural morphism $f_{q+1}f_qE \to f_{q+1}E$ is an equivalence, thus giving rise to a morphism $f_{q+1}E \to f_q{E}$. 
\begin{defn}
  The $q$-th slice of $E$, denoted $s_qE$, is the cofiber of the natural map $f_{q+1}E \to f_qE$. 
\end{defn}
\begin{rem}\label{rem:fibcofiberslice}
  By construction there is an equivalence $s_qE \simeq c_{q+1}f_qE$. It would also be reasonable (as is done in the equivariant slice filtration) to define $s_qE$ as the fiber of the map $c_{q+1}E \to c_{q+1}c_{q}E \simeq c_qE$. Since this fiber is also easily seen to be equivalent to $c_{q+1}f_qE \simeq s_qE$, it makes no difference.  
\end{rem}

\begin{comment}
Using \Cref{lem:recognitionconn} we get the following.
\begin{lem}\label{lem:recslices}
  Suppose there is a fiber sequence 
  \[
F_{q+1} \to f_qE \to S_{q+1}
  \]
such that $F_{q+1} \in \cC_{q+1}$ and $S_{q+1} \in (\cC)_{q+1}^{\perp}$. Then the canonical maps $F_{q+1} \to f_{q+1}f_qE \simeq f_{q+1}E$ and $S_{q+1} \to c_{q+1}f_qE \simeq s_qE$ are equivalences. 
\end{lem}
\end{comment}
The following is an immediate consequence of \Cref{lem:colim}. 
\begin{lem}\label{lem:colimsq}
 Both $f_q$ and $s_q$ commute with filtered colimits. 
\end{lem}
Once again, if $\cC_q$ is a stable subcategory of $\cC$, then $f_q$ and $s_q$ commute with all colimits. 

Finally, we note the following simple, but useful, result. 
\begin{lem}\label{lem:invertiblecomm}
  Suppose that $A \in \cC$ is invertible and has the following property: If $X \in \cC_q$, then $A \otimes X \in \cC_{q+k}$ for some fixed integer $k \in \Z$, independent of $q$. Then, $f_q(A \otimes E) \simeq A \otimes f_{q-k}E$ for any $E \in \cC$, and similar for $s_q(A \otimes E)$. 
\end{lem}
\begin{proof}
 Suppose $X \in \cC_q$. It follows that there are equivalences
  \begin{align*}
\Hom_{\cC}(X,f_q(A^{-1} \otimes E)) &\simeq \Hom_{\cC}(X,A^{-1} \otimes E) \\
& \simeq \Hom_{\cC}(X \otimes A,E)\\
& \simeq \Hom_{\cC}(X \otimes A,f_{q+k}E) \\
& \simeq \Hom_{\cC}(X,A^{-1} \otimes f_{q+k}E).
\end{align*}
By the Yoneda lemma we see that $f_q(A^{-1} \otimes E) \simeq A^{-1} \otimes f_{q+k}E$ - it follows that if $X \in \cC_q$, then $A^{-1} \otimes X \in \cC_{q-k}$ for all $q$.  
Now running the same argument with $A \otimes E$ instead of $A^{-1} \otimes E$ shows that $f_q(A \otimes E) \simeq A \otimes f_{q-k}E$. The result for $s_q(A \otimes E)$ follows from the defining cofiber sequences.
\end{proof}
\subsection{Pelaez's theorem}\label{sec:pelaez}
In the motivic category, Pelaez \cite{pelaez_functor} studied the behavior of the slice filtration under pullback. His results generalize to our setting, giving very general criteria for when slices commute with functors. 

Recall that a functor $F\colon \cC \to \cD$ between stable $\infty$-categories is called exact if $F$ carries zero objects into zero objects and preserves fiber sequences. Now let $\cC$ and $\cD$ be stably monoidal categories with slice filtrations $\{ \cC_i \}$ and $\{ \cD_i \}$ respectively, and let $F \colon \cC \to \cD$ be an exact functor. We are interested in the relationship between $F(f_q^{\cC}E)$ and $f_q^{\cD}F(E)$, and similar for $F(s_q^{\cC}E)$ and $s_q^{\cD}F(E)$. The simplest case is when the functor has a particularly nice left adjoint. 
\begin{lem}\label{lem:adjointfunctor}
  Suppose $G \colon \cC \to \cD$ is an exact functor between stably monoidal categories with slice filtrations satisfying $G(\cC_q) \subseteq \cD_q$ for all $q \in \Z$. Suppose moreover that $G$ has a left adjoint $F$ such that $F(\cD_q) \subseteq \cC_q$ for all $q \in \Z$. Then, there are natural equivalences $\alpha_q \colon G(f_q^{\cC}E) \xr{\sim} f_q^{\cD}G(E)$ and $\beta_q \colon G(s_q^{\cC}E) \xr{\sim} s_q^{\cD}G(E)$ for all $E \in \cC$ and $q \in \Z$. 
\end{lem}
\begin{proof}
  We first prove the result for $G(f_q^{\cC}E)$ - the defining cofiber sequences then show the corresponding result for $G(s_q^{\cC}E)$. This is a simple consequence of \Cref{lem:recognitionconn}. Indeed, we have a cofiber sequence 
  \[
G(f_q^{\cC}E) \to G(E) \to G(c_q^{\cC}E)
  \]
  for which we need to show that $G(f_q^{\cC}E) \in \cD_q$ and $G(c_q^{\cC}E) \in \cD^{\perp}_q$. The first follows by assumption (since ($f_q^{\cC}E \in \cC_q$), while for the second we have
  \[
\Hom_{\cD}(X,G(c_q^{\cC}E)) \simeq \Hom_{\cC}(F(X),c_q^{\cC}E) \simeq \ast
  \]
  for any $X \in \cD_q$, since $c_q^{\cC}E \in \cC_q^{\perp}$ by construction. It follows that $G(c_q^{\cC}E) \in \cD_q^{\perp}$ as required. 
\end{proof}

\begin{exmp}[Slices and base change]
  Given a stably monoidal category $\cC$, there exists an $\infty$-category $\CAlg(\cC)$   of commutative algebra objects in $\cC$, see \cite[Chapter 2]{ha}. Given such an $A \in \CAlg(\cC)$ we can form the category $\Mod_{\cC}(A)$ of $A$-modules in $\cC$, which is a stably monoidal category $A$ with the relative $A$-linear tensor product \cite[Section 4.5]{ha}. 

The following is not hard to verify. 
\begin{lem}\label{lem:modules}
  Given a slice filtration $\{ C_i \}$ on $\cC$, there exists a slice filtration on $\Mod_{\cC}(A)$, defined by letting $\Mod_{\cC}(A)_i$ be the smallest full subcategory of $\Mod_{\cC}(A)$ generated under colimits and extensions by $\cK_i \otimes A$ for each compact generator $\cK_i$ of $\cC_i$. 
\end{lem}
We will write $f_q^{A}$ and $s_q^{A}$ for the corresponding functors. There is an adjoint pair
\[
\xymatrix{
  - \otimes A \colon \cC \ar@<0.5ex>[r] & \ar[l] \Mod_{\cC}(A) \colon U
}
\]
where $U$ denotes the forgetful functor, and we will apply \Cref{lem:adjointfunctor} to the functor $U$. 
\begin{lem}
  Let $A$ be a commutative algebra object in $\cC$ and assume that $A \in \cC_0$, then there are equivalences
    \begin{align*}
\alpha_q \colon U(f_q^{A}E) \simeq f_{q}^{\cC}U(E) \\
\beta_q \colon U(s_q^{A}E) \simeq s_{q}^{\cC}U(E)
  \end{align*}
  for any $E \in \Mod_{\cC}(A)$. 
\end{lem}
\begin{proof}
   We need to show that $U(f_q^{A}E) \in \cC_q$ and $M \otimes A \in \cD_q$ whenever $M \in \cC_q$. The latter is clear from the definition of $\cD_q$ and the fact that tensor products commute with colimits. For the former, note that $\cC_q$ is generated under colimits and extensions by $\cK_q \otimes A$, and it is easily seen to be enough to check that $U(\cK_q \otimes A) \in \cC_q$ (since $U$ preserves colimits). Since $U(\cK_q \otimes A) \simeq \cK_q \otimes A$, it follows that we must check that $\cK_q \otimes A \in \cC_q$. But since $\cK_q \in \cC_q$ and $A \in \cC_0$, we have $\cK_q \otimes A \in \cC_q$ by \ref{a6}. 
\end{proof}
\begin{rem}
  In the motivic context this example is well known, see for example \cite[Lemma 2.1(4)]{levine_tripathi}.
\end{rem}
\end{exmp}
\begin{exmp}[Multiple slice filtrations]\label{exmp:multslices}
  Given a stably monoidal category $\cC$, there can be many different slice filtrations. Suppose we are given two slice filtrations $\{\cC_q\}$ and $\{\widetilde{\cC}_q\}$ on $\cC$ with corresponding functors $f_q$ and $\widetilde{f}_q$. Assume that $\widetilde{\cC}_i \subseteq \cC_i$ for all $i \in \Z$. By \Cref{lem:recognitionconn} (note that $\cC_i^\perp \subseteq \widetilde{\cC}_i^{\perp}$) we see that if $f_qE \in \widetilde{\cC}_q$ for all $q \in \Z$, then there are equivalences $f_qE \xr{\simeq} \widetilde{f}_qE$ and $s_qE \xr{\simeq} \widetilde{s}_qE$. This can also be proved by using the identity functor in \Cref{lem:adjointfunctor}. 
\end{exmp}

When we study the interaction between slices and Betti realization, \Cref{lem:adjointfunctor} will not suffice, and we need a stronger result, which we base on a theorem of Pelaez \cite{pelaez_functor}.  We begin with the following lemma, which is also well known in the motivic context. 
\begin{lem}\label{lem:colimdecomp}
Let $\cC$ be a stably monoidal category with a slice filtration, then the canonical morphism 
  \[
\phi \colon \colim_{p \le q} f_pE \xrightarrow{} E 
  \]
  is an equivalence for any $q \in \Z$. 
\end{lem}
\begin{proof}
 Let $\cal{G}$ be a set of compact generators of $\cC$. It suffices to prove that $\Hom_{\cC}(g,\phi)$ is an equivalence for each $g \in \cal{G}$. By \ref{a5} each $g$ is contained in some $\cC_i$, and so it suffices to show that $\Hom_{\cC}(g,f_i\phi)$ is an equivalence for suitable $i$. By \Cref{lem:colimsq} we are reduced to showing that 
  \[
\Hom_{\cC}(g,\colim_{p \le q}f_if_pE) \to \Hom_{\cC}(g,f_iE)
  \]
is an equivalence. But it is easy to see that the natural transformation $f_i(\epsilon) \colon f_if_p \to f_i \text{id} \simeq f_i$ is a natural equivalence whenever $p \le i$, and the result follows.  
\end{proof}
We now define the precise conditions that will be used in Pelaez's theorem.  Throughout this paper, we will consistently require the following conditions:
  \begin{enumerate}[align =left, label=Condition (\arabic*):]
    \item $F(E) \stackrel{\eqref{lem:colimdecomp}}{\simeq} F (\colim_{p \le q} f^{\cC}_pE) \simeq \colim_{p \le q} F(f^{\cC}_pE)$ \label{cond1}.
    \item $F(f_q^{\cC}E) \in \cD_{q}$ \label{cond2}.
    \item $F(s^{\cC}_qE) \in \cD_{q+1}^{\perp}$ \label{cond3}. 
  \end{enumerate}

  Observe that the counit adjunctions give rise to morphisms, 
  \[
\xymatrix{ f_{q}^{\cD}F(E) & \ar[l]_-{f_q^{D}F(\varepsilon)} f_{q}^{\cD}Ff^{\cC}_{q}(E) \ar[r]^-{\tilde \varepsilon } & F f^{\cC}_{q}(E)}
  \]
for any $E \in \cC$. If \hyperref[cond2]{Condition (2)} is satisfied for all $q \in \Z$, then $\widetilde{\varepsilon}$ is an equivalence for all $q \in \Z$, and hence there are morphisms
\[
\alpha_q \colon F(f^{\cC}_qE) \to f_{q}^{\cD}F(E). 
\]
\begin{comment}
  Alternatively, if $\eta$ is an equivalence for all $q$, then we denote the resulting morphism
\[
\widetilde \alpha_q \colon f_{q}^{\cD}F(E) \to Ff_q^{\cC}(E). 
\]
\end{comment}
By the definition of $s_q$ there is an induced morphism
\[
\beta_q \colon F(s^{\cC}_q(E)) \to s_{q}^{\cD}F(E).
\]
making the following diagram of cofiber sequences commute:
\[\setlength\mathsurround{0pt}
\begin{tikzcd}
  F(f_{q+1}^{\cC}E) \arrow{r} \arrow{d}{\alpha_{q+1}} & F(f_{q}^{\cC}E) \arrow{d}{\alpha_q} \arrow{r} & F(s_q^{\cC}E) \arrow{d}{\beta_q}\arrow{r} & \Sigma F(f_{q+1}^{\cC}E) \arrow{d}{\Sigma \alpha_{q+1}} \\
  f_{q+1}^{\cD}F(E) \arrow{r} & f_q^{\cD}F(E) \arrow{r} & s_q^{\cD}F(E) \arrow{r} & \Sigma f_{q+1}^{\cD}F(E).
\end{tikzcd}
\]
\begin{defn}
  We say that $F \colon (\cC,\cC_q) \to (\cD,\cD_q)$ is compatible with the slice filtration at $E$ if $\alpha_q$ and $\beta_q$ are equivalences for all $q \in \Z$. If the slices filtrations on $\cC$ and $\cD$ are understood, then we will simply say that $F \colon \cC \to \cD$ is compatible with the slice filtration at $E$.
\end{defn}
Note that this implies that the towers $F(f_{q+1}^{\cC}E)$ are $f_{q+1}^{\cD}F(E)$ are equivalent. Pelaez's theorem then gives precise conditions to ensure that an exact functor $F$ is compatible with the slice filtration at a given object $E \in \cC$. 
\begin{thm}[Pelaez]\label{thm:pelaez}
  Suppose that $F \colon \cC \to \cD$ is an exact functor between stably monoidal categories with slices filtrations $\{ \cC_i \}$ and $\{ \cal{D}_i \}$ respectively. If $F$ satisfies \hyperref[cond1]{Condition (1)}, \hyperref[cond2]{Condition (2)}, and \hyperref[cond3]{Condition (3)} for $E \in \cC$, then $F \colon (\cC,\cC_i) \to (\cD,\cD_i)$ is compatible with the slice filtration at $E \in \cC$.  
\end{thm}
\begin{proof}(Pelaez)
  We sketch the proof, since it is essentially the same as that given by Pelaez \cite[Theorem 2.12]{pelaez_functor} (see also the thesis of Kelly \cite[Section 4.2.2]{kelly_thesis}). We first note that the conditions of the theorem imply that the natural maps $\alpha_q$ and $\beta_q$ do exist. By the argument given in \cite[Lemma 2.10]{pelaez_functor} we have equivalences
       \[
\alpha_{q+1}(f^{\cC}_qE) \colon F (f^{\cC}_{q+1}(f^{\cC}_qE)) \xr{\simeq} f^{\cD}_{q+1}F(f^{\cC}_q E)
     \]
     and
     \[
\beta_{q}(f^{\cC}_qE) \colon F (s^{\cC}_q(f^{\cC}_qE)) \xr{\simeq} s_{q}^{\cD} F(f^{\cC}_q E)
     \]
     for all $q \in \Z$. 

     We now proceed to show the result for $s_q$, since essentially the same proof works for $f_q$. By \Cref{lem:colimdecomp} we have $E \simeq \colim_{p \le q}f_pE$. By \hyperref[cond1]{Condition (1)} and \Cref{lem:colimsq} the morphism $\beta_q(E)$ is given by $\colim_{p \le q}\beta_q(f^{\cC}_pE)$, and so it suffices to show that $\beta_q(f_p^{\cC}E)$ is an equivalence for all $p \le q$. By the discussion in the first paragraph it is true for $p = q$. The result then follows by a downward induction exactly as done by Pelaez. 
\end{proof}

\begin{rem}
   In the situation of \Cref{lem:adjointfunctor}, we have that \hyperref[cond2]{Condition (2)} and \hyperref[cond3]{Condition (3)} hold automatically - the former by assumption, and the latter by an easy adjunction argument. However, there seems to be no reason for \hyperref[cond1]{Condition (1)} to hold in general. The extra condition appears to arise because the assumptions in \Cref{lem:adjointfunctor} are stronger - if $F(c_{c+1}^{\cC}E) \in \cD_{q+1}^{\perp}$, then $F(s_{q}^{\cC}E) \in \cD_{q+1}^{\perp}$, but the converse need not be true. 
 \end{rem}

\section{The effective and cellular effective motivic slice filtrations}\label{sec:eff}
\subsection{The stable and cellular motivic homotopy category}
Let $\SH(S)$ denote the Morel--Voevodsky's stable motivic homotopy category over a base scheme $S$ \cite{morel_voe_a1}. We assume that $S$ is a Noetherian scheme of finite Krull dimension, and that $S$ is essentially smooth over a field of characteristic exponent $c$, i.e., $c=1$ when the characteristic is zero, and is the characteristic otherwise.\footnote{Some of the results in this paper can be extended e.g., to smooth schemes over Dedekind domains using work of Spitzweck \cite{1207.4078}, but we leave the details to the interested reader.}

By \cite[Corollary 1.2]{robalo} the $\infty$-category underlying this category (which we also denote by $\SH(S)$) is a stably symmetric monoidal category in our terminology. We let $\unit = \Sigma^{\infty}_+S$ denote the monoidal unit of this category. If $\Sm/S$ denotes the category of separated smooth schemes of finite type of $S$, then the set of objects
\[
\{ \Sigma^{p,q}\Sigma^{\infty}X_+\mid p,q \in \Z, X \in \Sm/S\}
\]
are a set of compact generators of $\SH(S)$ \cite[Theorem 9.1]{di_cell}. 

It is useful to consider the \emph{cellular} motivic category, as defined by Dugger--Isaksen \cite{di_cell}. 
\begin{defn}
  The cellular motivic category $\SH(S)_{\cell}$ is the localizing subcategory of $\SH(S)$ generated by ${\Sigma^{p,q}\unit}$ for all ${p,q \in \Z}$. A spectrum is called cellular if it lies in $\SH(S)_{\cell}$. 
\end{defn}
This subcategory is a stable presentable $\infty$-category by \cite[Proposition 1.4.4.11]{ha}. Since the tensor product in $\SH(S)$ commutes with colimits in both variables, it is easy to check that the tensor product of two cellular motivic spectra is again cellular, and hence that $\SH(S)_{\cell}$ is a stably monoidal category. 

Such cellular spectra include $\KGL$, the motivic spectrum representing algebraic $K$-theory, $\MGL$, the algebraic cobordism spectrum (both are proved in \cite{di_cell}), $\M A[1/c]$, the $c$-inverted motivic Eilenberg--Maclane spectrum associated to a abelian group $A$ \cite{hoyois_hhm}, and $\KQ$, the motivic spectrum representing Hermitian $K$-theory \cite{cell_KQ} (here we required that the base scheme has no points of characteristic two).

An argument similar to \Cref{lem:colim} shows that the inclusion $\SH(S)_{\cell} \subset \SH(S)$ has a cocontinuous right adjoint, which we denote $\Cell$. On the level of homotopy categories, this is equivalent to the functor studied by Dugger--Isaksen in \cite{di_cell}. 

Since we will need it later, we introduce a closely related subcategory.
\begin{defn}
    Let $S$ be an essentially smooth scheme over a field of characteristic exponent $c$. We say that $E \in \SH(S)$ is $c$-cellular if it is in the localizing subcategory generated by $\Sigma^{a,b}\unit[1/c]$ for $a,b \in \Z$. 
\end{defn}
\begin{rem}
  Clearly, if the characteristic exponent is 1, then $c$-cellular objects are simply cellular objects. In general, every $c$-cellular object is cellular, but the converse need not be true. 
\end{rem}

\subsection{Quotients and localizations of $\MGL$}\label{sec:quotlocMGL}
Of fundamental importance to us will be the algebraic cobordism spectrum $\MGL$. We briefly recall its construction here, referring the reader to \cite{voe_a1} for more details. Let $BGL_n$ denote the classifying space of the group scheme $GL_n$ over $S$. There is a universal bundle $\gamma_n^{\text{mot}} \to BGL_n$, and we let $MGL_n =(BGL_n)^{\gamma_n^{\text{mot}}}$ be the motivic Thom spectrum associated to this bundle. The canonical inclusion $BGL_n \to BGL_{n+1}$ and standard properties of Thom spectra gives rise to a morphism $\mathbb{P}^1\otimes MGL_{n} \to MGL_{n+1}$, and hence a motivic spectrum
\[
\MGL = (MGL_0,MGL_1,MGL_2,\ldots)
\]
Applying \cite[Lemma 6.1]{di_cell} there is an equivalence
\begin{equation}\label{eq:MGL}
\MGL \simeq  \colim_n \Sigma^{-2n,-n}\Sigma^{\infty}MGL_n \simeq \colim_n \Sigma^{-2n,-n}\Sigma^{\infty}(BGL_n)^{\gamma_n^{\text{mot}}}. 
\end{equation}

The spectrum $\MGL$ and certain quotients and localization of it, will be studied extensively in the sequel, and so we begin by defining exactly the spectra that we need, following \cite{levine_tripathi}. 

Recall that there is a classifying map $L \to \MGL_{\ast,\ast}$, where $L \cong \Z[a_1,a_2,\cdots]$ is the Lazard ring, see \cite[Sec.6.1]{hoyois_hhm} or \cite[Corollary 6.7]{nso_landweber} (here the grading is such that $a_i$ has bidegree $(2i,i)$). We will implicitly identify elements of $L$ with elements of $\MGL_{\ast,\ast}$. In fact, this is not such an abuse of notation; if $S$ is the spectrum of a field, then the map $L[1/c] \to \MGL_{(2,1)\ast}[1/c]$ is an equivalence \cite[Proposition 8.2]{hoyois_hhm}. 

First, we define $\MGL/a_i$ as the cofiber of $a_i \colon \Sigma^{2i,i}\MGL \to \MGL$. Given a finite collection $\{ i_1,i_2,\ldots, i_k\} \subset \mathbb{N}$ we define $\MGL/(a_{i_1},\ldots,a_{i_k})$ inductively by 
\[
\MGL/(a_{i_1},\ldots,a_{i_k}) = N/a_{i_k}
\]
where $N \simeq \MGL/(a_{i_1},\ldots,a_{i_{k-1}})$ (here the quotient $N/a_{i_k}$ is defined in the obvious way). By \cite[Remark 1.5]{levine_tripathi} this is equivalent to the $\MGL$-module
\[
\MGL/(a_{i_1},\ldots,a_{i_k}) = \MGL/a_{i_1} \otimes_{\MGL} \cdots \otimes_{\MGL} \MGL/a_{i_k}. 
\]

For an arbitrary subset $\mathcal{I} \subset \mathbb{N}$ we define
\[
\MGL/(\mathcal{I}) = \colim\limits_{\{ i_1, \ldots, i_k \} \subset \mathcal{I}} \MGL/(a_{i_1},\ldots,a_{i_k})
\]
where the colimit is taken over the filtered poset of finite subsets of $\mathcal{I}$. 

Now let $\cal{I}^{c}$ be the complement of $\cal{I}$, and let $\Z[\cal{I}^c]$ denote the graded polynomial ring on the $a_i,i \in \cal{I}^{c}$. Let $\cal{I}_0$ denote a collection of homogeneous elements of $\Z[\cal{I}^c]$, and define $\MGL/(\cal{I})[\cal{I}_0^{-1}] = \MGL/(\cal{I})[\{z_j^{-1} \mid z_j \in \cal{I}_0\}]$.

We will also have need to consider $p$-local and mod $p$-versions. Let $\MGL_{(p)}$ be the $p$-localization of $\MGL$. Explicitly, this can be given by the colimit of the maps
\[
\MGL \xr{n} \MGL
\]
where $n$ is an integer relatively prime to $p$. Then, for an arbitrary subset $\mathcal{I} \subset \N$ we define $\MGL_{(p)}/(\cal{I})$ as above.  Similarly, we can also define $\MGL_{(p)}/(\cal{I})[\cal{I}_0^{-1}] = \MGL_{(p)}/(\cal{I})[\{z_j^{-1}\}\mid z_j \in \cal{I}_0 \}]$. Finally, we write $\MGL/(\cal{I},p)[\cal{I}_0^{-1}]$ for the cofiber of the multiplication by $p$ map on $\MGL_{(p)}/(\cal{I})[\cal{I}_0^{-1}]$.

It will be useful to introduce terminology to describe these type of spectra. 
\begin{defn}\label{def:chromatictype}
  A motivic spectrum is said to be a localized quotient of $\MGL$ if it can be constructed by quotients and localization of $\MGL$ or $\MGL_{(p)}$ as above.  
\end{defn}
\begin{exmp}
  The following examples show that these give many analogs of spectra familiar in chromatic homotopy. Following \cite{levine_tripathi} define subsets
\[
\begin{split}
  B_p^c &= \{ a_i \mid i\ne p^k-1, k \ge 0\},\\
  B\langle n \rangle_p^c &= \{ a_i \mid i \ne p^k-1, 0 \le k \le n \},\\
  k \langle n \rangle_p^c & = \{ a_i \mid i \ne p^n-1 \}.
\end{split}
\]
These give rise to the following motivic spectra:
\[
\begin{split}
  \BP & = \MGL_{(p)}/(\{a_i \mid i \in B_p^c \}), \\
  \BP\langle n \rangle & = \MGL_{(p)}/(\{a_i \mid i \in B\langle n \rangle_p^c \}), \\
  \E(n) & = \BP \langle n \rangle[a_{p^n-1}^{-1}] \\
  \k(n) &= \MGL_{(p)}/(\{ a_i \mid i \in k \langle n \rangle_p^c \}) \\
  \K(n) &= \k(n)[a_{p^n-1}^{-1}]. 
\end{split}
\]
We note that $\BP$ and $\E(n)$ are Landweber exact over $\MGL$ but the other spectra constructed are not. 
\end{exmp}
Finally, we will need the following technical result on the slices of quotients of $\MGL$. 
\begin{prop}[Levine--Tripathi]\label{prop:lttechnicalassumption}
  Let $\cal{I} \subset \N$ be arbitrary, and let $\cal{I}^c$ be the complement of $\cal{I}$. Let $M$ denote either $\MGL/(\cal{I}), \MGL_{(p)}/(\cal{I})$ or $\MGL_{(p)}/(\cal{I},p)$, then there is an equivalence
  \[
s_0(M) \simeq M/(\cal{I}^{c}).
\]
\end{prop}
\begin{proof}
  The case where $M = \MGL_{(p)}/(\cal{I})$ is shown in the proof of \cite[Proposition 4.5]{levine_tripathi}, and the case $M = \MGL/(\cal{I})$ can be proved in the same way. The mod $p$ case then follows since the functor $s_0$ is exact. 
\end{proof}
\subsection{Effective and cellular effective slices}\label{sec:cellular}

Following Voevodsky \cite{voe_open}, consider the collection
\[\cK_q = \{ \Sigma^{a,b}\Sigma^\infty X_+ \mid X \in \Sm/S , a \in \Z, b \ge q \} \subseteq \SH(S) .\] 
Let $\Sigma^{q}_T \SH(S)^{\eff} \subseteq \SH(S)$ denote the localizing subcategory generated by $\cK_q$. It is then easy to see that this forms a slice filtration of $\SH(S)$ in the sense of \Cref{defn:slicefiltration}. Indeed, \ref{a1}-\ref{a5} follow immediately, and \ref{a6} can be checked on generators, for which it is seen to be true.\footnote{In fact, by \cite{grso_slices}, it even satisfies the property that if $X \in \cC_i$ and $Y \in \cC_j$, then $X \otimes Y \in \cC_{i+j}$.}  
The filtration of $\SH(S)$ by the $q$-th connective covers
\[
\cdots \subset \Sigma^{q+1}_T \SH(S)^{\eff} \subset \Sigma^{q}_T \SH(S)^{\eff} \subset \cdots
\]
is Voevodsky's slice filtration, which we call the effective motivic slice filtration. We let $f_q^{S}$ and $s_q^{S}$ denote the associated functors, although we will omit the superscript unless it is unclear. 

As noted in the introduction, many motivic spectra are cellular. The analog of the effective slice filtration in $\SH(S)_{\cell}$ is defined by the collection  collection
\[
\cK^{\cell}_q = \{ \Sigma^{a,b}\unit \mid, a \in \Z, b \ge q \} \subseteq \SH(S)_{\cell}.\] 
In particular, if we let $\Sigma^q_T\SH(S)^{\eff}_{\cell} \subseteq \SH(S)_{\cell}$ denote the localizing subcategory of $\SH(S)_{\cell}$ generated by $\cK^{\cell}_q$ we get the cellular effective motivic slice filtration 
\[
\cdots \Sigma^{q+1}_T \SH(S)^{\eff}_{\cell} \subset \Sigma^{q}_T \SH(S)^{\eff}_{\cell} \subset \cdots.
\]
This forms a slice filtration of $\SH(S)_{\cell}$. We let $f_q^{\Cell}$ and $s_q^{\Cell}$ denote the associated functors. The following simple result is very useful. 
\begin{lem}\label{lem:shifts}
  For any $a,b \in \Z$ and any motivic spectrum $E$, we have $f_q(\Sigma^{a,b}E)\simeq \Sigma^{a,b}f_{q-b}E$, and similar for $s_q,f_q^{\Cell}$ and $s_q^{\Cell}$. 
\end{lem}
\begin{proof}
  Apply \Cref{lem:invertiblecomm} with $A = \Sigma^{a,b} \unit$. 
\end{proof}

\subsection{The comparison theorem}
Our main result in this section is to compare the effective and cellular effective slice filtrations. The motivation for this arises in later sections, where we will compare the motivic and $C_2$-equivariant categories, because we understand precisely the behavior of the motivic spheres under equivariant Betti realization.  

We start with the following lemma.
\begin{lem}\label{lem:ccellularslices}
    Let $S$ be an essentially smooth scheme over a field of characteristic exponent $c$. Then, the slices of any $c$-cellular spectrum are cellular. 
\end{lem}
\begin{proof}
  By a localizing subcategory argument it suffices to show the result for $\Sigma^{a,b}\unit[1/c]$ and $a,b \in \Z$. By \Cref{lem:shifts} we can reduce further to checking the statement for $\unit[1/c]$ itself. 

  Let $E_2^{s,t}(\MU) = \Ext^{s,t}_{\MU_*\MU}(\MU_*,\MU_*)$ be the cohomology of the Hopf algebroid associated to complex cobordism. By \cite[Theorem 8.7]{hoyois_hhm} (following the ideas of Levine and Voevodsky) there is an equivalence
  \begin{equation}\label{eq:slicessphere}
s_q(\unit[1/c]) \simeq \bigvee_{s \ge 0} \Sigma^{2q-s,q} \M\Z[1/c] \otimes E_2^{s,2q}(\MU). 
  \end{equation}
  Note that by \cite[Proposition 2.2]{zahler_anss} $\Ext_{\MU_*\MU}^{s,2q}(\MU_*,\MU_*)$ is a finite group when $(s,q) \ne (0,0)$ (in which case it isomorphic to $\Z$). The result then follows from \cite[Proposition 8.1]{hoyois_hhm}. 
\end{proof}

\begin{lem}\label{prop:cell}
  Let $E$ be a motivic spectrum. If $\Cell (f_qE) \in \Sigma^{q}_T\SH(S)^{\eff}_{\cell}$ for all $q \in \Z$, then $\Cell \colon (\SH(S),\Sigma^q_T \SH(S)^{\eff}) \to (\SH(S)_{\cell},\Sigma^q_T \SH(S)^{\eff}_{\cell})$ is compatible with the slice filtration at $E$, i.e., there are equivalences
  \[
\xymatrix@R=0.6pc{
\alpha_q(E) \colon \Cell(f_qE) \ar[r]^-{\sim}& f_q^{\Cell}\Cell(E) \\
\beta_q(E) \colon \Cell(s_qE) \ar[r]^-{\sim}&  s_q^{\Cell}\Cell(E)}
 \]
for all $q \in \Z$. 
\end{lem}
\begin{proof}
This is a consequence of \Cref{lem:adjointfunctor}. Indeed, the left adjoint of $\Cell$ is the inclusion functor, and by definition $\Sigma^q_T \SH(S)_{\cell}^{\eff} \subseteq \Sigma_T^q \SH(S)^{\eff}$. 
\end{proof}
Of fundamental importance for us is the following result \cite[Theorem 5.7]{mot_twisted}. 
\begin{thm}[Spitzweck--{\O}stv{\ae}r]\label{thm:somglslice}
  The algebraic cobordism theorem $\MGL$ is in $\SH(S)^{\eff}_{\cell}$. 
\end{thm}
\begin{proof}
  Consider the unit map $\unit \to \MGL$. The proof of Spitzweck--{\O}stv\ae r shows that the cofiber of this map is contained in $\Sigma^1_T \SH(S)_{\cell}^{\eff} \subset \SH(S)_{\cell}^{\eff}$. Since $\unit \in \SH(S)_{\cell}^{\eff} $, and this category is closed under extensions, the result follows. 
\end{proof}
\begin{cor}\label{cor:cellularquotients}
\begin{enumerate}
  \item   For any $0 \le n \le \infty$ the quotient $\MGL/(a_1,\ldots,a_n) \in \SH(S)^{\eff}_{\cell}$
  \item For any simplicial $\Z[1/c]$-module $A$ the motivic Eilenberg--Maclane spectrum $\mathbf{M}A \in \SH (S)^{\eff}_{\cell}$.  
\end{enumerate}

\end{cor}
\begin{proof}
(1) immediate in light of the previous theorem. The case $A = \Z[1/c]$ of (2) is then a consequence of the equivalence $\MGL/(a_1,a_2,\ldots)[1/c] \simeq \M\Z[1/c]$ \cite[Theorem 7.12]{hoyois_hhm}, while the general case follows from \cite[Proposition 4.13]{hoyois_hhm}.
\end{proof}
\begin{rem}\label{rem:cellularquotients}
  In fact, the proof of Spitzweck--{\O}stv\ae r shows that $\MGL$ lies in the smallest full subcategory closed under colimits and extensions generated by $\{ \Sigma^{a,b}\unit \mid a,b \ge 0 \}$, see the remark at the bottom of page 586 of \cite{mot_twisted}. We will study this category, denoted $\SH(S)^{\veff}_{\cell}$ in more detail in the next section. For now, we note that the same arguments show that (1) and (2) hold with $\SH (S)^{\veff}_{\cell}$ in place of $\SH (S)^{\eff}_{\cell}$. 
\end{rem}

Note that by \Cref{lem:ccellularslices} $f_q(\MGL[1/c])$ is cellular, so that we do not need to apply $\Cell$ in the following lemma. 
\begin{lem}\label{lem:slicemgl}
  Let $S$ be an essentially smooth scheme over a field of characteristic exponent $c$. Then, $\Cell \colon (\SH(S),\Sigma^q_T \SH(S)^{\eff}) \to (\SH(S)_{\cell},\Sigma^q_T \SH(S)^{\eff}_{\cell})$ is compatible with the slice filtration at $\MGL[1/c]$, i.e, there are equivalences
  \[ 
\begin{split} 
\xymatrix{\alpha_q \colon f_q(\MGL[1/c]) \ar[r]^-{\sim} & f_q^{\Cell}(\MGL[1/c])} \\
\xymatrix{\beta_q \colon s_q(\MGL[1/c]) \ar[r]^-{\sim} & s_q^{\Cell}(\MGL[1/c])} 
\end{split}
  \]
  for all $q \in \Z$. 
\end{lem}
\begin{proof}\sloppy
  By \Cref{prop:cell} it is enough to show that $\Cell(f_q\MGL[1/c]) \in \Sigma^q_T\SH(S)^{\eff}_{\cell}$.  By \cite[Theorem 7.12]{hoyois_hhm} and the proof of \cite[Theorem 4.7]{spitzweck_slices}, $f_q\MGL[1/c]$ is the colimit of a diagram of $\MGL[1/c]$-modules of the form $\Sigma^{2k,k}\MGL[1/c]$ where $k \ge q$.\footnote{As noted in \cite[Theorem 8.5]{hoyois_hhm} one can remove Spitzweck's assumption that the field is perfect. Indeed, for an essential smooth morphism $f \colon T \to S$ of schemes, \cite[Theorem 2.12]{pelaez_functor} shows that $f^*f^S_q(\MGL_S[1/c]) \simeq f^T_qf^*(\MGL_S[1/c]) \cong f_q^T(\MGL_T[1/c])$.}  Since $\Cell(\Sigma^{2k,k}\MGL[1/c]) \simeq \Sigma^{2k,k}\MGL[1/c] \in \Sigma_T^k\SH_{\cell}^{\eff} \subseteq \Sigma^q_T\SH^{\eff}_{\cell}$ by \Cref{thm:somglslice}, the result follows because $\Sigma^q_T \SH(S)_{\cell}^{\eff}$ is closed under colimits. 
\end{proof}
Using this result we can now prove compatibility of the slice filtration with $\Cell$ in the case that $E$ is Landweber exact in the sense of \cite{nso_landweber}. Note that these spectra are always cellular by \cite[Proposition 8.4]{nso_landweber}. We will write $LB = L[b_1,b_2,\ldots]$, so that the Hopf algebroid $(L,LB)$ is the Hopf algebroid classifying formal group laws and strict isomorphisms. 
\begin{lem}\label{lem:slicele}
  Let $M_*$ be a Landweber exact $L[1/c]$-module, and $E \in \SH(S)$ the associated motivic spectrum. Then, $\Cell \colon (\SH(S),\Sigma^q_T \SH(S)^{\eff}) \to (\SH(S)_{\cell},\Sigma^q_T \SH(S)^{\eff}_{\cell})$ is compatible with the slice tower at $E$, i.e., there are equivalences
  \[  
   \xymatrix{\alpha_q \colon f_q(E) \ar[r]^-{\sim} &f_q^{\Cell}(E) }
\]
and
\[
   \xymatrix{\beta_q \colon s_q(E) \ar@<0.5ex>[r]^-{\sim} & s_q^{\Cell}(E)}  
  \]
  for all $q \in \Z$. 
\end{lem}
\begin{proof}
  This is essentially the argument given in \cite[Lemma 8.11]{hoyois_hhm}. If $M_*$ is flat, then it is a filtered colimit of finite sums of shifts of $L[1/c]$ by Lazard's theorem, and the corresponding spectrum $E$ is a filtered colimit of the corresponding diagram of $\MGL[1/c]$-modules. Since $f_q$ commutes with filtered colimits by \Cref{lem:colimsq}, the previous lemma implies the result holds for such $E$ (using, say, \Cref{lem:shifts} to handle the suspensions). For the general case, it suffices to show that $\Cell f_q(\MGL[1/c] \otimes E) \in \Sigma^{q}_T\SH(S)^{\eff}_{\cell}$, since $E$ is a retract of $\MGL[1/c] \otimes E$. But $\MGL[1/c] \otimes E$ is the spectrum associated to the flat Landweber exact $L[1/c]$-module $LB[1/c] \otimes_{L[1/c]} M_*$. 
\end{proof}

We now give a full computation of the cellular slices of the motivic sphere spectrum. This is done using precisely the techniques of Voevodsky, Levine, and others, in computing $s_q(\unit)$.  We again let $E_2^{s,q}(\MU)$ denote the $E_2$ term of the Adams--Novikov spectral sequence, i.e., $E_2^{s,q}(\MU) \cong \Ext^{s,q}_{\MU_*\MU}(\MU_*,\MU_*)$. 
\begin{thm}\label{thm:cellslicesunit}
   Let $S$ be an essentially smooth scheme over a field of characteristic exponent $c$. Then, there is an equivalence
  \[
  s_q^{\Cell}(\unit[1/c]) \simeq s_q(\unit[1/c]) \simeq \bigvee_{s \ge 0} \Sigma^{2q-s,q}\M\Z[1/c] \otimes E_2^{s,2q}(\MU). 
  \]
\begin{proof}
  We will assume that $c = 1$ for legibility. Consider the diagram in $\SH(S)_{\cell}$
\[
\setlength\mathsurround{0pt}
\begin{tikzcd}
  \unit \arrow{r} & \MGL \arrow[shift right=0.5ex]{r} \arrow[shift left=0.5ex]{r} & \MGL^{\otimes 2}  \arrow[shift right=1ex]{r} \arrow[shift left=1ex]{r} \arrow{r} & \cdots
\end{tikzcd}
\]
which gives rise to a morphism
\begin{equation}\label{eq:totalization}
\setlength\mathsurround{0pt}
\begin{tikzcd}
   s_q^{\Cell}(\unit) \arrow{r}  &\underset{\Delta}{\widetilde \lim} \,s_q^{\Cell}(\MGL^{\otimes \bullet}) 
\end{tikzcd}
\end{equation}
Here we have used the notation ${\widetilde \lim}$ for the limit in $\SH_{\cell}$. It is easy to see that ${\widetilde \lim}(-) \simeq \Cell {\lim}(-)$, by checking that the latter satisfies the universal property of the limit in $\SH_{\cell}$. 

We claim that \eqref{eq:totalization} is an equivalence. The proof of this is the same as \cite[Proposition 2.9]{rso_slices}, and so we simply sketch it for the reader. In fact, more generally, it is true that 
\[
s^{\Cell}_q(\overline \MGL^{\otimes m}) \xr{\simeq} \underset{\Delta}{\widetilde \lim} s_q^{\Cell}(\overline{\MGL}^{\otimes m} \otimes \MGL^{\otimes(\bullet + 1)}),
\]
with the claimed equivalence being the case $m = 0$. 

For $m \ge q+1$ the claim follows because both sides are trivial, since $\overline \MGL \in \Sigma^{1}_T \SH^{\eff}_{\cell}$. 
We now induct downwards on $m$ as in \emph{loc.~cit.}, using the commutative diagram of cofiber sequences
\[
\setlength\mathsurround{0pt}
\begin{tikzcd}
  s_q^{\Cell}(\overline{\MGL}^{\otimes(m+1)}) \arrow{d} \arrow{r} & \Cell \displaystyle\lim_{\Delta} s_q^{\Cell}(\overline{\MGL}^{\otimes (m+1)} \arrow{d} \otimes \MGL^{\otimes(\bullet + 1)}) \\
  s_q^{\Cell}(\overline{\MGL}^{\otimes m}) \arrow{d} \arrow{r} & \Cell \displaystyle\lim_{\Delta} s_q^{\Cell}(\overline{\MGL}^{\otimes m} \otimes \MGL^{\otimes(\bullet + 1)}) \arrow{d}\\
  s_q^{\Cell}(\overline{\MGL}^{\otimes m} \otimes \MGL) \arrow{r} & \Cell \displaystyle\lim_{\Delta} s_q^{\Cell}(\overline{\MGL}^{\otimes m} \otimes \MGL \otimes \MGL^{\otimes(\bullet + 1)})
\end{tikzcd}
\]
Here the top horizontal arrow is inductively an equivalence, and the bottom is via \cite[Lemma 2.10]{rso_slices}. It follows that the middle horizontal arrow is also an equivalence, and it follows that \eqref{eq:totalization} is an equivalence as claimed. 

Since $\MGL^{\otimes \bullet}$ is Landweber exact, we have $s_q^{\Cell}(\MGL^{\otimes \bullet}) \simeq s_q(\MGL^{\otimes \bullet})$ by \Cref{lem:slicele}, and by naturality this is compatible with the maps in the totalization. It follows that 
\[
s_q^{\Cell}(\unit) \simeq \Cell \lim_{\Delta} s_q^{\Cell}(\MGL^{\otimes \bullet}) \simeq \Cell \lim_{\Delta} s_q(\MGL^{\otimes \bullet}) \simeq \Cell s_q(\unit),
\]
where the last equivalence follows from the computation of the slices of the sphere spectrum, cf \cite[Proposition 2.9]{rso_slices}. By \Cref{lem:ccellularslices}, $\Cell s_q(\unit) \simeq s_q(\unit)$, and so we conclude that $s_q^{\Cell}(\unit) \simeq s_q(\unit)$.
\end{proof}

\end{thm}
We are now in a position to prove the main result of this section. 
\begin{thm}\label{thm:effcellcompare}
   Let $S$ be an essentially smooth scheme over a field of characteristic exponent $c$, then $\iota \colon (\SH(S)_{\cell},\Sigma^q_T\SH(S)^{\eff}_{\cell}) \to (\SH(S),\Sigma^q_T\SH(S)^{\eff})$ is compatible with the slice filtration at a $c$-cellular motivic spectrum $E$, i.e, there are equivalences 
  \[  
  \begin{split}
    \xymatrix{ \alpha_q \colon \iota f^{\Cell}_q(E) \ar[r]^-{\sim} & f_q(\iota E)}\\
    \xymatrix{ \beta_q \colon  \iota s^{\Cell}_q(E) \ar[r]^-{\sim} & s_q(\iota E)} 
  \end{split}
  \]
  for all $q \in \Z$. 
\end{thm}
\begin{proof}
We want to apply \Cref{thm:pelaez}. It is clear that \hyperref[cond1]{Condition (1)} and \hyperref[cond2]{Condition (2)} hold, so that such maps $\alpha_q$ and $\beta_q$ as in the statement of the theorem do exist. A localizing subcategory argument then reduces to checking the statement on $\Sigma^{p,q}\unit[1/c]$, and as usual by shifting, we can reduce to checking it on $\unit[1/c]$ itself. It thus suffices to see that $s_q^{\Cell}(\unit[1/c]) \in( \Sigma^{q+1}_T\SH(S))^{\perp}$.   But \Cref{thm:cellslicesunit} shows that $s_q^{\Cell}(\unit[1/c]) \simeq s_q(\unit[1/c])$, and by construction the latter is always in $(\Sigma^{q+1}_T\SH(S))^{\perp}$. 
  \end{proof}
\begin{rem}
  We do not know if the stronger condition that $\Cell \colon (\SH(S),\Sigma^q_T\SH(S)^{\eff}) \to (\SH(S)_{\cell},\Sigma^q_T\SH(S)^{\eff}_{\cell})$ is compatible with the slice filtration at a motivic spectrum $E$ holds in general. The work of this section shows that it holds whenever $E$ is $c$-cellular. 
\end{rem}
\section{The very effective cellular slice filtration}\label{sec:veff}
The effective slice filtration considered previously constructs slices by filtering with respect to the Tate sphere $\G_m \simeq S^{1,1}$. An alternative was introduced in \cite{mot_twisted} by Spitzweck--{\O}stv{\ae}r which filters with respect to $\mathbb{P}^1 \simeq S^{2,1}$. To be precise, consider the collection of objects
\[\cK_q = \{ \Sigma^{2a,a}\Sigma^\infty_+X \mid  X \in \Sm/S, a \ge q\} \subseteq \SH(S).\] 
Define $\Sigma^{q}_T \SH(S)_{\cell}^{\eff}$ to the the smallest full subcategory of $\SH(S)_{\cell}$ generated under colimits and extensions by $\cal{K}_q$. The full subcategories $\{ \Sigma^q_T \SH(S)_{\cell}^{\eff} \}_{q \in \Z}$ define a slice filtration of $\SH(S)_{\cell}$. We denote the associated slice functors by $\widetilde{f}_q$ and $\widetilde{s}_q$.  The very effective slice filtration has been studied in more detail by Bachmann in \cite{1610.01346}. 

As in the previous section we can define a cellular version of this filtration via the collection
\[\cK^{\cell}_q = \{ \Sigma^{2a,a}\unit \mid, a \ge q \} \subseteq \SH(S)_{\cell}.\] 
This gives rise to the very effective cellular slice filtration of $\SH(S)_{\cell}$
\[
\cdots \Sigma^{q+1}_T \SH(S)^{\veff}_{\cell} \subset \Sigma^{q}_T \SH(S)^{\veff}_{\cell} \subset \cdots.
\]
Note that these categories are \emph{not} localizing. We write $\widetilde f^{\Cell}_q$ and $\widetilde s^{\Cell}_q$ for the associated functors, which are not triangulated.

The following can be proved via \Cref{lem:invertiblecomm}.  
\begin{lem}\label{lem:veffectivetsuspension}
  For a motivic spectrum $E$, we have $\widetilde{f}^{\Cell}_q(\Sigma^{2a,a} E) \simeq \Sigma^{2a,a} \widetilde{f}^{\Cell}_{q-a}(E)$ and similar for $\widetilde{s}_q^{\Cell}$. 
\end{lem}

Since we always have $\Sigma^q_T\SH(S)^{\veff}_{\cell} \subset \Sigma^q_T \SH(S)^{\eff}_{\cell}$ there exists a natural transformation $\widetilde{f}^{\Cell}_q \to {f}^{\Cell}_q$.  The following is a consequence of \Cref{exmp:multslices}. 
\begin{lem}\label{lem:effvsveff}
  If $E$ is a cellular motivic spectrum such that $f^{\Cell}_q(E) \in \Sigma^q_{T}\SH(S)^{\veff}_{\cell}$ for all $q \in \Z$, then there are equivalences 
  \[
  \setlength\mathsurround{0pt}
  \begin{tikzcd}
\widetilde{f}^{\Cell}_q(E) \arrow{r}{\simeq} &   f^{\Cell}_q(E) 
  \end{tikzcd}
  \]
  and
  \[  \setlength\mathsurround{0pt}
  \begin{tikzcd}
\widetilde{s}^{\Cell}_q(E) \arrow{r}{\simeq} &   s^{\Cell}_q(E).
  \end{tikzcd}
  \]
\end{lem}
This gives the following. 
\begin{prop}\label{thm:effvsveff}
  Let $E$ be a localized quotient of $\MGL$, then
    \[\setlength\mathsurround{0pt}
    \begin{tikzcd}    
f_q(E) \arrow{r}{\simeq} &f^{\Cell}_q(E) \arrow{r}{\simeq} &  \widetilde f^{\Cell}_q(E) 
  \end{tikzcd}
  \]
  and
  \[  \setlength\mathsurround{0pt}
  \begin{tikzcd}
s_q(E) \arrow{r}{\simeq} &s^{\Cell}_q(E) \arrow{r}{\simeq} &  \widetilde s^{\Cell}_q(E)
  \end{tikzcd}
  \]
  for all $q \in \Z$. 
\end{prop}
\begin{proof}
  Since the motivic spectra considered are always cellular, the left hand equivalences follow from \Cref{thm:effcellcompare}. To show the theorem for $\MGL$ we must show that $f_q(\MGL) \in \Sigma^q_T\SH(S)^{\veff}_{\cell}$. However, we have already seen in the proof of \Cref{lem:slicemgl} that $f_q(\MGL)$ is a colimit of a diagram of $\MGL$-modules of the form $\Sigma^{2k,k}\MGL$ where $k \ge q$.  We know that $\MGL \in \SH(F)^{\veff}_{\cell}$ by \Cref{rem:cellularquotients}, and hence that $\Sigma^{2k,k}\MGL \in \Sigma_T^k\SH(F)^{\veff}_{\cell} \subseteq \Sigma_T^q\SH(F)^{\veff}_{\cell}$. Since this category is closed under colimits we are done.

It follows by the defining cofiber sequences that $f_q(\MGL/(a_{i_1},\ldots,a_{i_n}))$ is in $\Sigma^q_T\SH(S)^{\veff}_{\cell}$ for any finite subset $\{ i_i,\ldots, i_n\} \subset \N$, so the proposition is true for $\MGL/(a_{i_1},\ldots,a_{i_n})$. Similar arguments work in the $p$-local and mod $p$-case. For an arbitrary subset $\cal{I}$ we defined $\MGL/(\cal{I})$ as a filtered colimit of terms of the form $\MGL/(a_{i_1},\ldots,a_{i_n})$. It follows that the proposition holds for $\MGL/(\cal{I})$. 

  To deal with the localizations we appeal to work of Levine and Tripathi. In particular, let $\cal{I} \subset \N$ be arbitrary, and let $M$ denote either $\MGL/(\cal{I})$, $\MGL_{(p)}/(\cal{I})$ or $\MGL/(\cal{I},p)$. By \Cref{prop:lttechnicalassumption} the assumptions of \cite[Corollary 2.4]{levine_tripathi} are satisfied, and the proof of Levine and Tripathi shows that $f_q(M[\cal{I}_0^{-1}])$ is a colimit of terms of the form $\Sigma^{-2k,-k}f_{q+k}M$ for $k \ge 0$. The first part of this proposition shows that $f_{q+k}M \in \Sigma^{q+k}_T\SH(S)^{\veff}_{\cell}$, and hence that $\Sigma^{-2k,-k}f_{q+k}M \in \Sigma^{q}_T\SH(S)^{\veff}_{\cell}$. The result follows since this category is closed under colimits.  
  \end{proof}

It is remarked in \cite{mot_twisted}, and then proved in detail in \cite{1610.01346} that the very effective slice filtration is the positive part of a $t$-structure on the very effective motivic stable homotopy category. In the case that $S$ is the spectrum of a perfect field $F$, Bachmann used the homotopy $t$-structure to give a description of the very effective motivic category in terms of homotopy sheaves. We will give a similar description for the very effective cellular slice filtration. We begin with the cellular analog of the homotopy $t$-structure, as defined in \cite[Section 2.1]{hoyois_hhm}. We restrict ourself to working over $\SH(F)$ where $F$ is a perfect field. 

\begin{defn}
  The category of cellularly $d$-connective objects (or just connective when $d=0$) is the smallest full-subcategory of $\SH(F)_{\cell}$ generated under colimits and extensions by the collection
\[
\{ \Sigma^{s,t}\unit \mid s-t \ge d \}. 
\]
\end{defn}
Note that $\widetilde f^{\Cell}_qE$ is always cellularly $q$-connective by definition. 

\begin{prop}\label{prop:cellconnectivity}
  A cellular motivic spectrum $E\in \SH(F)_{\cell}$ is cellularly connective if and only if $\pi_{a,b}E = 0$ for $a-b< 0$. 
\end{prop}
\begin{proof}
  Let $\cC$ denote the category of those cellular spectra $E$ such that $\pi_{a,b}E = 0$ for $a-b < 0$. We observe that $\Sigma^{s,t}\unit$ is in $\cC$ whenever $s \ge t$ by Morel's connectivity theorem \cite[Section 5.3]{morel_connectivity}. It is clear that $\cC$ is closed under extensions, so we must show that it is closed under arbitrary colimits. As in the proof of \cite[Lemma 5.10]{mot_twisted} we can assume the colimit is either a coproduct or a pushout, for which the result is clear. We thus see that all cellularly connective objects are in $\cC$. 

   For the converse, recall again that $\Sigma^{s,t}\unit$ is connective for $s \ge t$. Suppose now we are given a cellular spectrum $E$ with $\pi_{a,b}E = 0$ for $a-b<0$. Then, by the method of killing cells, as described in the proof of Proposition 4(2) of \cite{1610.01346} (see also the `Details on killings cells' after the proposition), one can construct a cellularly connective spectrum $Z$ along with a map $Z \to E$ inducing an isomorphism on bigraded homotopy groups. Since both $Z$ and and $E$ are cellular, this map is an equivalence \cite[Corollary 7.2]{di_cell}.
\end{proof}
\begin{rem}
  If $E$ is $d$-connective in the sense of \cite[Section 2.1]{hoyois_hhm}, then we observe that $\pi_{a,b}\Cell(E) \cong \pi_{a,b}E = 0$ for $a-b<d$ by \cite[Theorem 2.3]{hoyois_hhm}. It follows that $\Cell(E)$ is cellularly $d$-connective in the above sense.
\end{rem}
\begin{rem}
  A similar, but not equivalent, category is considered by Shkmebi--Isaksen \cite{isaksen_shkembi}, who only allow the spheres $\Sigma^{s,t}\unit$ for $s-t \ge 0$ and $s \ge 0$. Their version of \Cref{prop:cellconnectivity} is \cite[Proposition 3.17]{isaksen_shkembi} - the proof is in the same spirit as the one given above. 
\end{rem}

If we restrict to effective cellular spectra, then we end up with a similar characterization of very effective cellular spectra - we thank Tom Bachmann for suggesting that this holds. 
\begin{prop}\label{prop:veffcell}
  If $E \in \Sigma^q_T\SH(F)_{\cell}^{\eff}$, then $E \in \Sigma^q\SH(F)_{\cell}^{\veff}$ if and only if 
  \[
\pi_{a,b}E = 0
  \]
  for all $a,b \in \Z$ satisfying $a-b < q$ and $b \ge q$. 
\end{prop}
\begin{proof}
  This is identical to the previous proposition. We can reduce to the case $q = 0$ by shifting. Let $\cD$ denote the collection of those $E \in \SH(F)_{\cell}^{\eff}$ such that  $\pi_{a,b}E = 0$ for all $a,b \in \Z$ satisfying $a-b < 0$ and $b \ge 0$. This is closed under colimits and extension and contains $\Sigma^{2a,a}\unit$ for $ a \ge 0$, so that in particular $\SH(F)^{\veff}_{\cell} \subseteq \cD$. Conversely, if $E$ is in $\cD$ then by killing cells as above we can build $Z \in \SH(F)^{\veff}_{\cell}$ along with a map $Z \to E$ inducing an equivalence in bigraded homotopy groups (here we only need to use the spheres $\Sigma^{a,b}\unit$ with $a - b <0 $ and $b \ge 0$ because $E \in \SH(F)^{\eff}_{\cell}$ by assumption).   
\end{proof}
\begin{rem}
  The collection of cellularly connective objects forms the positive part of a $t$-structure on $\SH(F)_{\cell}$, and the collection of cellular very effective motivic spectra forms the positive part of a $t$-structure on $\SH(F)^{\eff}_{\cell}$. By the same argument as \cite[Proposition 4.3(3)]{1610.01346} the functor $r_0 \colon \SH(F)_{\cell} \to \SH(F)_{\cell}^{\eff}$ is $t$-exact. 
\end{rem}
\begin{cor}\label{cor:cellle}
  If $f_q^{\Cell}E$ is cellularly $q$-connective for all $q \in \Z$, then there are equivalences
  \[\setlength\mathsurround{0pt}
  \begin{tikzcd}
\widetilde{f}^{\Cell}_q(E) \arrow{r}{\simeq} &   f^{\Cell}_q(E) 
  \end{tikzcd}
  \]
  and
  \[\setlength\mathsurround{0pt}
  \begin{tikzcd}
\widetilde{s}^{\Cell}_q(E) \arrow{r}{\simeq} &   s^{\Cell}_q(E).
  \end{tikzcd}
  \]
\end{cor}
\begin{proof}
  This is a consequence of the previous proposition and \Cref{lem:effvsveff}. 
  \end{proof}
 The main reason for introducing cellularly $q$-effective spectra is to investigate the very effective cellular slices of a Landweber exact motivic spectrum. The analogous argument also works in the non-cellular situation, see \cite[Lemma 2.4 and Remark 2.5]{1712.01349}.
\begin{prop}\label{lem:prople}
  Let $E \in \SH(F)$ be a Landweber exact motivic spectrum. Then, there are equivalences
  \[\setlength\mathsurround{0pt}
  \begin{tikzcd}
f_q(E) \arrow{r}{\simeq} &f^{\Cell}_q(E) \arrow{r}{\simeq} &  \widetilde f^{\Cell}_q(E) 
  \end{tikzcd}
  \]
  and
  \[\setlength\mathsurround{0pt}
  \begin{tikzcd}
s_q(E) \arrow{r}{\simeq} &s^{\Cell}_q(E) \arrow{r}{\simeq} &  \widetilde s^{\Cell}_q(E)
  \end{tikzcd}
  \]
  for all $q \in \Z$. 
\end{prop}
\begin{proof}
  The exact same argument as given by Hoyois in \cite[Lemma 8.11]{hoyois_hhm} shows that $f_qE \simeq f_q^{\Cell}(E)$ is always cellularly $q$-connective. The result is then a consequence of \Cref{cor:cellle}. 
\end{proof}
\section{Equivariant slices}\label{sec:equivariant}
In the previous sections we have studied various filtrations on the stable motivic homotopy category. In this section, we study the Hill--Hopkins--Ravenel slice filtration in $C_2$-equivariant homotopy theory, and show how Betti realization interacts with the different slice filtrations. 
\subsection{The $C_2$-equivariant homotopy category}
In this section we give a brief introduction to the category of $C_2$-equivariant spectra. For a more detailed discussion, one can see the appendix of \cite{hhr}, or for a slightly shorter introduction, see \cite[Section 2]{hill_meier}.

We define $\SH(C_2)$ to be the $\infty$-category associated to the symmetric monoidal category of orthogonal $C_2$-spectra \cite{mandell_may}. This has been studied in some detail in \cite[Section 5]{mnn_15}. In particular, it is a stably monoidal category, see Definition 5.10 and Remark 5.12 of \cite{mnn_15}. It has a set of compact generators, given by $\{S^0, \Sigma^{\infty}_+ C_2 \}$. 

We will grade homotopy groups by the real representation ring $RO(C_2) = \{ a + b \sigma \mid a,b \in \Z \}$, where $\sigma$ denotes the sign representation.  We follow motivic notation, and write $S^{p,q\sigma}$ for the smash product $(S^1)^{p-q} \otimes S^{q \sigma}$ (if $q = 0$, then we will sometimes just write $S^p$). We let $\pi^{C_2}_{p,q}E = [S^{p,q\sigma},E]^{C_2}$, and write $\pi^{e}_{p,q}E$ for the underlying non-equivariant groups.  We recall that the equivariant and non-equivariant homotopy groups of a $C_2$-spectrum $E$ can be combined into a Mackey functor, which we denote by $\underline{\pi}_{p,q}E$. 

In the $C_2$-equivariant category, the analog of the motivic algebraic cobordism spectrum is the real bordism spectrum $M\R$ constructed by Araki, Landweber, and Hu--Kriz \cite{hk_real}. This is a real oriented ring spectrum, and hence admits a map $\pi_{2k}\MU \cong L_{2k} \to \pi_{2k,k}^{C_2}M\R$ from the Lazard ring, which is in fact an isomorphism \cite[Theorem 2.28]{hk_real}. It follows that we can form quotients and localization of $M\R$ analogous to those considered for $\MGL$ in \Cref{sec:quotlocMGL}, and hence there is a notion of a localized quotient of $M\R$. 
\begin{exmp}\label{exmp:realktheory}
  The spectrum $K\R$ defined by $a_1^{-1}M\R/(a_2,a_2,\ldots)$ agrees with Atiyah's real $K$-theory spectrum \cite[Theorem 3.18]{hk_real}. 
\end{exmp}
\subsection{The Hill--Hopkins--Ravenel slice filtration}

The Hill--Hopkins--Ravenel slice filtration starts by defining the following \emph{slice cells}:
\begin{enumerate}
  \item $S^{2q,q\sigma}$ of dimension $2q$, 
  \item $S^{2q-1,q\sigma}$ of dimension $2q-1$, and
  \item $S^{q} \otimes (C_2)_+$ of dimension $q$. 
\end{enumerate}
Note that the dimension always corresponds to dimension of the underlying non-equivariant sphere. We then define $\Sigma^k\SH(C_2)$ to be the smallest full subcategory of $\SH(C_2)$ closed under extensions and colimits containing the slice cells of dimension $\ge q$. This gives rise to the Hill--Hopkins--Ravenel equivariant slice filtration
\[
\cdots \subset \Sigma^{q+1} \SH(C_2)^{\hhr} \subset \Sigma^{q} \SH(C_2)^{\hhr} \subset \cdots.
\]
Note that these categories are not localizing (in that sense, they are closer to the very effective motivic slice filtration than the effective one)

Following standard equivariant notation we denote $P_{q}X$ for the associated colocalization functor, and $P^{q-1}X$ for the localization functors, and $P_q^qX$ for the $n$-th slice; that is, there are functorial fiber sequences
\[
P_{q}X \to X \to P^{q-1}X
\]
and
\[
P_{q+1}X \to P_qX \to P_q^qX
\]
where $P_qX \in \Sigma^q \SH(C_2)^{\hhr}$, $P^{q-1}X \in (\Sigma^q \SH(C_2)^{\hhr})^{\perp}$, and $P_q^qX \in (\Sigma^{q+1}\SH(C_2)^{\hhr})^{\perp}$. 
\begin{rem}\label{rem:hhrvsreg}
  The regular slice filtration, first introduced by Ullman \cite{ullman_thesis,ullman_agt} uses only the first and third slice cells above. Let us denote the full subcategories generated under colimits and extensions by the regular slice cells of dimension $\ge q$ by $\Sigma^q \SH(C_2)^{\reg}$, and the associated functors by $\widetilde{P}_q,\widetilde{P}^q$, and $\widetilde{P}^q_q$. 
  It is easy to see \cite[Prop.3-4]{ullman_agt} (or as a consequence of \Cref{exmp:multslices}) that if $P_qX \in \Sigma^q \SH(C_2)^{\reg}$ for all $q \in \Z$, then there are equivalences $\widetilde{P}_qX \xr{\simeq} P_qX$, and $\widetilde{P}_q^qX \xr{\simeq} P_q^qX$. As we shall see, this is always the case when the slice tower of a motivic spectrum is compatible with Betti realization.  
\end{rem}
We say that a spectrum $X$ is a $q$-slice if $P_q^qX \simeq X$. In particular, if $X$ is a $q$-slice, then $X \in (\Sigma^{q+1}\SH(C_2)^{\hhr})^{\perp}$. It is known that a spectrum is a 0-slice if and only if it is of the form $H\underline{M}$ for $\underline{M}$ a Mackey functor whose restriction maps are all monomorphisms \cite[Proposition 4.50(ii)]{hhr}. Since $P^{q+2k}_{q+2k}(\Sigma^{2k,k\sigma}X) \simeq \Sigma^{2k,k\sigma}P_q^q(X)$ (\cite[Theorem 2.18]{hill_primer} or as a consequence of \Cref{lem:invertiblecomm}) we deduce the following. 
\begin{prop}\label{prop:HMorthogonal}
  Let $\underline{M}$ be a constant Mackey functor, then
  \[
  \Sigma^{2q,q\sigma}H\underline{M} \in (\Sigma^{2q+1}\SH(C_2)^{\hhr})^{\perp}.\] 
\end{prop}

For a $C_2$-equivariant spectrum $E$, the odd and even slices can be completely described in terms of $\underline{\pi}_{\ast,\ast}(E)$.  Given a $C_2$-Mackey functor $\underline{M}$, we let $P^0\underline{M}$ denote the maximal quotient of $\underline{M}$ for which the restriction map $\underline{M}(C_2/C_2) \to \underline{M}(C_2/e)$ is a monomorphism. The following is then a combination of \cite[Proposition 4.20 and Lemma 4.23]{hhr} and \cite[Corollary 2.16]{hill_primer}.
\begin{prop}\label{prop:c2slices}
  For a $C_2$-equivariant spectrum $E$ the slices are given by 
  \[
P_{2q}^{2q}(E) \simeq  \Sigma^{2q,q\sigma} HP^0 \underline\pi_{2q,q}(E)
  \quad \text{and} \quad P_{2q-1}^{2q-1}(E) \simeq \Sigma^{2q-1,q\sigma}H\underline \pi_{2q-1,q}E.
  \]
  \end{prop}
Because of this, Hill and Meier \cite{hill_meier} introduced the notion of even and strongly even $C_2$-equivariant spectra. 

\begin{defn}[Hill--Meier]
  A $C_2$-spectrum $E$ is even if $\underline \pi_{2q-1,q}E = 0$ for all $q$. It is strongly even if it is even and $\underline \pi_{2q,q}E$ is a constant Mackey functor for all $q \in \Z$, i.e., if the restriction
  \[
\pi_{2q,q}^{C_2}E \to \pi_{2q,q}^eE \cong \pi_{2q}E
  \]
is an isomorphism. 
\end{defn}

In particular, if $E$ is even, then the odd slices vanish (and conversely). Many spectra, such as $BP\R$ and $M\R$ are known to be even (in fact, strongly even) \cite{hill_meier}. Because of this we will study variants of the Hill--Hopkins--Ravenel slice tower, which we refer to as the even and odd slices towers. 

 For a general $C_2$-equivariant spectrum $E$, we say that the \emph{even slice tower} is the tower
\[ \setlength\mathsurround{0pt}
\begin{tikzcd}
 \cdots \arrow{r} & P_{2q+2}E \arrow{d} \arrow{r} & \arrow{d} P_{2q}E \arrow{r} & \arrow{d} P_{2q-2}E \arrow{r} & \cdots \\
  & \overline{P_{2q+2}^{2q+2}}E & \overline{P_{2q}^{2q}}E & \overline{P_{2q-2}^{2q-2}}E,
\end{tikzcd}
\]
where $P_{2q+2}E \to P_{2q}E$ is the composite $P_{2q+2}E \to P_{2q+1}E \to P_{2q}E$, and each $P_{2q+2}E \to P_{2q}E \to \overline{P_{2q}^{2q}}E$ is a cofiber sequence. We call the collection $\{ \Sigma^{2q} \SH(C_2)^{\hhr} \}_{q \in \Z}$ the even slice filtration. Clearly if a motivic spectrum is even, then the even slice tower is just the usual regular slice tower where we have removed the (contractible) odd slices. Similarly, we can define the \emph{odd slice tower} and the odd slice filtration by using only the odd $P_{2q+1}$. 

\subsection{Betti realization and slices}

Given a scheme over $\mathbb{R}$, the associated analytic space $X(\mathbb{C})$ is a $C_2$-space, where the $C_2$-action arises from complex conjugation. As shown in \cite[Section 4]{heller_ormsby} for example, this leads to a stable realization functor $\SH(\R) \to \SH(C_2)$. Given $k \subseteq \R$, we can compose with the base-change functor $\SH(k) \to \SH(\R)$ to define a realization functor, which we denote by $\Re$. We record its fundamental properties \cite[Proposition 4.8]{heller_ormsby}. 
\begin{thm}(Heller--Ormsby)\label{thm:bettirealization}
  Let $k \subseteq \R$ be a field, then there is a strong symmetric monoidal, cocontinuous functor $Re \colon \SH(k) \to \SH(C_2)$. Moreover, the functor satisfies $\Re(S^{p,q}) \simeq S^{p,q\sigma}$. 
\end{thm}
\sloppy By composing with the inclusion functor, Betti realization also gives rise to a functor $\Re \colon \SH(k)_{\cell} \to \SH(C_2)$. 
\begin{rem}
  The adjunction constructed by Heller and Ormsby is a Quillen adjunction between model categories. The main result of \cite{mazel_gee} shows that this gives rise to an adjunction between the associated $\infty$-categories. 
\end{rem}
The following is a folklore result, for which we are unable to find a reference in the literature.
\begin{prop}
  The $C_2$-equivariant realization of $\MGL$ is real bordism $\mathrm{M}\mathbb{R}$. 
\end{prop}
\begin{proof}
Let $\gamma_n^{\text{mot}}$ denote the universal bundle over the Grassmannian $BGL_n$, so that 
\[
\MGL \simeq \colim_n \Sigma^{-2n,-n} \Sigma^\infty (BGL_n)^{\gamma_n^{\text{mot}}}
\] 
By \Cref{thm:bettirealization} we have 
\[
\Re(\MGL) \simeq \colim_n \Sigma^{-2n,-n\sigma} \Re(\Sigma^{\infty} (BGL_n^{\gamma_n^{\text{mot}}}))
\]
By construction, the realization functor is constructed levelwise, and so it commutes with the suspension spectrum. Moreover, by construction, $\Re(BGL_n) \simeq BGL_n(\C) \simeq BU(n)$ equipped with its natural complex conjugation $C_2$-action. Using \Cref{thm:bettirealization} again, it is then easy to see that $\Re(BGL_n^{\gamma_n^{\text{mot}}}) \simeq BU(n)^{\gamma_n^{\text{real}}}$. Together we see that
\[
\Re(\MGL) \simeq \colim_n \Sigma^{-2n,n\sigma} \Sigma^\infty (BU(n))^{\gamma_n^{\text{real}}},
\]
which is a model of $M\R$ by \cite[B.252]{hhr}.
\end{proof}
As noted, Hu and Kriz showed that the restriction maps $M\R_{2k,k} \to MU_{2k}$ are isomorphisms and so we can define localized quotients of $M\R$ analogous to the localized quotients of $\MGL$ constructed in \Cref{sec:quotlocMGL}. Given a localized quotient $E^{\text{mot}}$ of $\MGL$, we will denote by $E^{\text{equiv}}$ the corresponding localized quotient of $M\R$. The following is then a consequence of the fact that realization commutes with colimits. 
\begin{cor}
  The Betti realization of a localized quotient $E^{\text{mot}}$ of $\MGL$ is the corresponding localized quotient $E^{\text{equiv}}$ of $M\R$. 
\end{cor}
\begin{rem}
  In the case where $E^{\text{mot}} = \BP\langle n \rangle$, this confirms the assumption made in \cite[Remark 3.15]{mbp_rationals}.
\end{rem}
We also have the following \cite[Theorem 4.17]{heller_ormsby}. 
\begin{prop}(Heller--Ormsby)\label{prop:bettiMA}
  Let $A$ be an abelian group, then $\Re(\M A) \simeq H\underline{A}$. 
\end{prop}

While simple, the following result motivates our use of the very effective cellular category, as it is not true for the effective cellular category. For example, since $S^{-1,0} \in \Sigma_T^0\SH(S)^{\eff}_{\cell}$, we have $S^{-1,0} \simeq f_q^{\Cell}(S^{-1,0})$, but the realization of $S^{-1,0}$ is not in $\Sigma^0 \SH(C_2)^{\reg}$. 
\begin{lem}\label{lem:cond2pel}
  Let $E \in \SH(k)_{\cell}$, then $\Re(\widetilde{f}_k^{\Cell}E) \in \Sigma^{2k} \SH(C_2)^{\reg} \subset \Sigma^{2k} \SH(C_2)^{\hhr}$. 
\end{lem}
\begin{proof}
  Consider the full subcategory $\cal{U} \subseteq \SH(S)_{\cell}$ consisting of those objects  $U \in \SH(S)_{\cell}$ for which $\Re(U) \in  \Sigma^{2k} \SH(C_2)^{\reg}$. The subcategory $\cal{U}$ is clearly thick, and because $\Re$ preserves colimits (\Cref{thm:bettirealization}) and $\Sigma^{2k}\SH(C_2)^{\reg}$ is closed under colimits, we see that $\cal{U}$ is closed under colimits as well. Moreover, \Cref{thm:bettirealization} again implies that the spheres $S^{2a,a}$ for $a > k$ are in $\cal{U}$. It follows that $\Sigma_T^k\SH(S)_{\cell}^{\veff} \subseteq \cal{U}$, and hence
  \[
\Re(\Sigma^k_T\SH(S)^{\veff}_{\cell}) \subseteq \Re(\cal{U}) \subseteq \Sigma^{2k} \SH(C_2)^{\reg}.
  \]
  In particular,  for $E \in \SH(k)_{\cell}$, we have $\Re(\widetilde{f}_k^{\Cell}E) \in \Sigma^{2k} \SH(C_2)^{\reg}$, as required.
\end{proof}

We are now in a position to apply Pelaez's theorem to compare the Betti realization of the slice tower of a motivic spectrum to both the even and odd slice filtrations of its Betti realization. 
\begin{thm}\label{thm:veffcellc2compthm}
  Let $E \in \SH(k)_{\cell}$ be a cellular motivic spectrum. If $\Re(\widetilde{s}_q^{\Cell}E) \in (\Sigma^{2q+1}\SH(C_2)^{\hhr})^{\perp}$ for all $q \in \Z$, then:
  \begin{enumerate}
    \item The functors \[\Re \colon (\SH(k)_{\cell},\Sigma^q_T \SH(k)_{\cell}^{\veff}) \to (\SH(C_2),\Sigma^{2q}\SH(C_2)^{\hhr})\] and \[\Re \colon (\SH(k)_{\cell},\Sigma^q_T \SH(k)_{\cell}^{\veff}) \to (\SH(C_2),\Sigma^{2q-1}\SH(C_2)^{\hhr})\]
    are compatible with the slice filtration at $E$. In particular, there are equivalences 
    \[
\Re(\widetilde{f}^{\Cell}_qE) \xr{\simeq} P_{2q}\Re(E) \xr{\simeq} P_{2q-1}\Re(E)
    \]
    so that 
    \[
P_i^i\Re(E) \simeq 
\begin{cases}
\Re(\widetilde{s}_q^{\Cell}E) & i = 2q \\
0 & \text{otherwise}
\end{cases}
    \]
    for all $q$.
    \item The $C_2$-spectrum $\Re(E)$ is even, so that $\underline{\pi}_{2q-1,q}\Re(E) = 0$ for all $q \in \Z$. 
    \item The regular and Hill--Hopkins--Ravenel slice towers of $\Re(E)$ are equivalent. 
  \end{enumerate}
\end{thm}
\begin{proof}
We first apply Pelaez's theorem to the functor \[\Re \colon (\SH(k)_{\cell},\Sigma^q_T \SH(k)_{\cell}^{\veff}) \to (\SH(C_2),(\Sigma^{2q}\SH(C_2)^{\hhr})).\] 

By \Cref{thm:bettirealization} and \Cref{lem:cond2pel} we see that \hyperref[cond1]{Condition (1)} and \hyperref[cond2]{Condition (2)} are satisfied. \hyperref[cond3]{Condition (3)} is the statement that $\Re(\widetilde s_q^{\Cell}E) \in (\Sigma^{2q+2}\SH(C_2)^{\hhr})^{\perp}$. Since $\Sigma^{2q+2}\SH(C_2)^{\hhr} \subset \Sigma^{2q+1}\SH(C_2)^{\hhr}$, we have $(\Sigma^{2q+1}\SH(C_2)^{\hhr})^{\perp} \subset (\Sigma^{2q+2}\SH(C_2)^{\hhr})^{\perp}$. By assumption we have $\Re(\widetilde s_q^{\Cell}E) \in (\Sigma^{2q+1}\SH(C_2)^{\hhr})^{\perp}$, so that \hyperref[cond3]{Condition (3)} is satisfied. 

A similar argument works for the functor 
\[\Re \colon (\SH(k)_{\cell},\Sigma^q_T \SH(k)_{\cell}^{\veff}) \to (\SH(C_2),(\Sigma^{2q-1}\SH(C_2)^{\hhr})).\] 
Since $\Sigma^{2q}\SH(C_2)^{\hhr} \subset \Sigma^{2q-1}\SH(C_2)^{\hhr}$, we can again use \Cref{thm:bettirealization} and \Cref{lem:cond2pel} to see that \hyperref[cond1]{Condition (1)} and \hyperref[cond2]{Condition (2)} are satisfied, and \hyperref[cond3]{Condition (3)}  is precisely the assumption of the theorem. 

By \Cref{thm:pelaez} we deduce that there are equivalences 
\[
\Re(\widetilde f^{\Cell}_qE) \xr{\simeq} P_{2q} \Re(E) \quad \text{ and } \quad \Re(\widetilde s^{\Cell}_qE) \xr{\simeq} \overline {P^{2q}_{2q}} \Re(E) 
\]
as well as 
\[
\Re(\widetilde f^{\Cell}_qE) \xr{\simeq} P_{2q-1} \Re(E) \quad \text{ and } \quad \Re(\widetilde s^{\Cell}_qE) \xr{\simeq} \overline {P^{2q-1}_{2q-1}} \Re(E).
\]
It follows that the natural map $P_{2q}\Re(E) \to P_{2q-1}\Re(E)$ is an equivalence, so that there are equivalences $\Re(\widetilde{f}_qE) \xr{\simeq} P_{2q}\Re(E) \xr{\simeq} P_{2q-1}\Re(E)$, as claimed. 

By construction, for any motivic spectrum $E$, there is a cofiber sequence
\[
P_{q+1}^{q+1}\Re(E) \to \overline {P_{q}^{q}}\Re(E) \to P_q^q\Re(E).
\]
If $q=2j-1$ is odd, then the previous paragraphs show that $P_{2j-1}^{2j-1}\Re(E) \simeq 0$, and $\Re(\widetilde s_j^{\Cell}E) \xr{\simeq} P_{2j}^{2j}\Re(E) \xr{\simeq} \overline {P_{2j-1}^{2j-1}}\Re(E)$. Similarly, if $q=2j$ is even, then $\Re(\widetilde s_j^{\Cell}E) \xr{\simeq} \overline {P_{2j}^{2j}}\Re(E) \xr{\simeq} P_{2j}^{2j}\Re(E)$, and $P_{2j-1}^{2j-1}\Re(E) \simeq 0$. In either case we see that the claimed formula for $P_i^{i}\Re(E)$ holds, and we have proved (1). 

Since the odd slices of $\Re(E)$ are contractible, it is even by definition, and by \Cref{prop:c2slices} we must have $\underline{\pi}_{2q-1,q}\Re(E) = 0$. This proves (2). 

To see that (3) holds, we must show that $P_q\Re(E) \in \Sigma^q \SH(C_2)^{\reg}$. Suppose $q=2j$ is even, then $P_{2j}\Re(E) \simeq \Re(\widetilde{f}_{j}^{\Cell}E) \in \Sigma^{2j}\SH(C_2)^{\reg}$ by (1) and \Cref{lem:cond2pel}. Similarly, if $q=2j-1$ is odd, then $P_{2j-1}\Re(E) \simeq P_{2j}\Re(E) \simeq \Re(\widetilde f_j^{\Cell}E) \in \Sigma^{2j}\SH(C_2)^{\reg} \subset \Sigma^{2j-1}\SH(C_2)^{\reg}$. This proves (3).  
\end{proof}
The condition that $\Re(\widetilde{s}_q^{\Cell}E) \in (\Sigma^{2q+1}\SH(C_2)^{\hhr})^{\perp}$ is stated more succinctly in the terminology of \cite{hhr} by the requirement that $\Re(\widetilde{S}_q^{\Cell}(E))$ is slice $(2q+1)$-null. This is a statement that can be checked by the vanishing of certain equivariant homotopy groups, see \cite[Lemma 4.14]{hhr}, however in practice this can be difficult to check. We will instead rely mainly on the following. 
\begin{cor}\label{cor:celldecomp}
  Let $E \in \SH(k)_{\cell}$ be a motivic spectrum. If, for each $q \in \Z$ there is an equivalence $\widetilde{s}_q^{\Cell}E \simeq \bigvee_{i \in I}\Sigma^{2q,q}\M A_i$, where each $A_i$ is an abelian group and $I$ is a finite indexing set, then the conditions of \Cref{thm:veffcellc2compthm} are satisfied. 
\end{cor}
\begin{proof}
Using \Cref{thm:bettirealization,prop:bettiMA} we see that $\Re(\widetilde s_q^{\Cell}E) \simeq \bigvee_{i \in I} \Sigma^{2q,q\sigma}H\underline{A}_i$. By \Cref{prop:HMorthogonal} we have $\Sigma^{2q,q\sigma}H\underline{A}_i \in (\Sigma^{2q+1}\SH(C_2)^{\hhr})^{\perp}$. Since the indexing set $I$ is finite, products and coproducts in $\SH(C_2)$ agree, and since $(\Sigma^{2q+1}\SH(C_2)^{\hhr})^{\perp}$ is closed under limits, we conclude that $\Re( \bigvee_{i \in I}\Sigma^{2q,q}\M A_i) \in (\Sigma^{2q+1}\SH(C_2)^{\hhr})^{\perp}$ as required. 
\end{proof}

By combining the previous theorem with work in the previous section we can give conditions when the effective slice filtration on a motivic spectrum $E$ is compatible with the Hill--Hopkins--Ravenel slice filtration on its realization $\Re(E)$. 
\begin{thm}\label{thm:effvshhr}
  Let $E$ be a motivic spectrum that satisfies the following conditions:
  \begin{enumerate}
    \item $E$ is cellular. 
    \item $f_qE \in \Sigma^q_T \SH(S)_{\cell}^{\veff}$ for all $q \in Z$.
    \item $\Re(s_qE) \in (\Sigma^{2q+1}\SH(C_2)^{\hhr})^{\perp}$ for all $q \in Z$. 
  \end{enumerate}
  Then, $\Re \colon (\SH(k),\Sigma^q_T\SH(S)^{\eff}) \to (\SH(C_2),\Sigma^{2q}\SH(C_2)^{\hhr})$ is compatible with the slice filtration at $E$, the $C_2$-spectrum $\Re(E)$ is even, and the regular and Hill--Hopkins--Ravenel slice towers of $\Re(E)$ agree. 
\end{thm}
\begin{proof}
  Combine \Cref{thm:effcellcompare,thm:effvsveff} with \Cref{thm:veffcellc2compthm}. 
\end{proof}
\subsection{Some examples}
In the following we let $E^{\text{mot}}$ denote a localized quotient of $\MGL$, and we let $A$ denote either $\Z,\Z_{(p)}$ or $\Z/p$ depending on if $E$ is a quotient and localization of $\MGL$, of $\MGL_{(p)}$, or of $\MGL_{(p)}/p$ respectively. Moreover, we write $E^{\text{top}}$ and $E^{\text{equiv}}$ for the corresponding motivic and topological spectrum. 
\begin{thm}\label{thm:comparasionthmchromatic}
  Let $E^{\text{mot}}$ be a localized quotient of $\MGL$, then $E^{\text{equiv}}$ satisfies the conditions of \Cref{thm:effvshhr}. The odd slices of $E^{\text{equiv}}$ are contractible, and 
  if there are only finitely many monomials in degree $2q$, then the even slices of $E^{\text{equiv}}$ are given by
  \[
P_{2q}^{2q}(E^{\text{equiv}}) \simeq \bigvee_I \Sigma^{2q,q\sigma} H\underline{A}  
  \]
 where the wedge is indexed by monomials of degree $2q$ in $E_{2q}^{\text{top}}$.
\end{thm}
\begin{proof}
 We have already seen that $E^{\text{mot}}$ is cellular, and that $f_qE^{\text{mot}} \in \Sigma^q_T \SH(S)^{\veff}_{\cell}$, see \Cref{thm:effvsveff}, so that (1) and (2) of \Cref{thm:effvshhr} are satisfied. 

   By using either Theorem 2.3, Corollary 2.4, or Corollary 2.5 as well as Remark 2.6 of \cite{levine_tripathi}, we can compute that 
  \[
s_q(E^{\text{mot}}) \simeq \Sigma^{2q,q}\M A \otimes E_{2q}^{\text{top}} \simeq \bigvee_{I} \Sigma^{2q,q}\M A
  \]
  where the wedge is indexed by monomials of degree $2q$ in $E_{2q}^{\text{top}}$ (compare Section 4 of \cite{levine_tripathi}). The $q$-th slice is just $\Sigma^{2q,q}\M E_{2q}^{\text{top}}$, and so Condition (3) of \Cref{thm:effvshhr} is also satisfied by \Cref{prop:HMorthogonal,prop:bettiMA}. Finally, if there are only finitely many monomials of degree $2q$, then the indexing set $I$ is finite, so that \Cref{cor:celldecomp} applies and gives the decomposition of the slices. 
  \end{proof}
  \begin{exmp}
    Under our assumptions on $k$, this theorem applies to $E = \KGL$, since there is an equivalence $\KGL \simeq a_1^{-1}\MGL/(a_2,a_3,\ldots)$ \cite[Theorem 5.2]{spitzweck_slices}. By \Cref{exmp:realktheory}, the realization of $\KGL$ is Atiyah's real $K$-theory spectrum. We recover the computation of the slices of $K\R$ mentioned in the introduction, namely that 
    \[
P^{q}_q K\mathbb{R} \simeq \begin{cases}
  \Sigma^{q,q/2\sigma}H\underline{\Z} & \text{ $q$ is even} \\
  0 & \text{ otherwise.}
\end{cases}
    \]
  \end{exmp}
  \begin{cor}
    Any localized quotient $E^{\text{equiv}}$ of $M\R$ is even, i.e, satisfies $\underline{\pi}_{2k-1,k}E^{\text{equiv}} = 0$ for all $k \in \Z$. 
  \end{cor}
  \begin{rem}
    In \cite[Corollary 4.6]{greenlees_meier} Greenlees and Meier prove that equivariant spectra of the form $BP\R/I$ are strongly even. The above results give an independent proof that they are even, but the logic is different; they compute the relevant homotopy groups to show evenness, while we use knowledge of the slices. 
  \end{rem}
 Another example comes from motivic Landweber exact spectra. 
 \begin{thm}
   Let $E$ be a motivic Landweber exact spectrum, with $E^{\text{top}}$ the corresponding topological Landweber exact spectrum. Then, its realization $\Re(E)$ satisfies the conditions of \Cref{thm:effvshhr}. The even slices of $\Re(E)$ are given by 
   \[
P_{2q}^{2q}(\Re(E)) \simeq \Sigma^q_T H\underline{\pi}_{2q}E^{\text{top}}. 
   \]
   and the odd slices are contractible. 
 \end{thm}
 \begin{proof}
   This is similar to the previous theorem - motivic Landweber exact spectra are, by definition, cellular, and we know that $f_qE^{\text{mot}} \in \Sigma^q_T\SH(S)^{\veff}_{\cell}$ by \Cref{thm:effvsveff}. By \cite{spitzweck_landweber} we have $s_q(E) \simeq \Sigma^{2q,q}\M E_{2i}^{\text{top}}$. It follows that the conditions of \Cref{thm:effvshhr} are satisfied, and the result follows. 
 \end{proof}
\begin{rem}
  In general, the realization of a cellular motivic spectrum need not have any relation with the slices of its Betti realization. For example, since motivic slices commute with suspension, but the equivariant ones do not, one can check that (for example) $\Sigma^{k,0}\M \F_2$ when $k \ne 0$ is a counter-example. For a more interesting example, the $q$-th slice of the motivic sphere spectrum always has a summand $\Sigma^{q,q}\M \Z/2$ (under the equivalence of \Cref{eq:slicessphere}, this summand is generated by $\alpha_1^q$, see \cite[Corollary 2.13]{rso_slices}). By \cite[Corollary 2.20 and Corollary 2.21]{hill_primer} this has the wrong suspension to be a $m$-slice (for any $m \in \Z$) unless $q = 0$. A similar statement holds for the slice of $\KQ$ using the computation of its slices in \cite{ro_slices_kq}. 

  The examples of $\unit$ and $\KQ$ given above should not be surprising - after all, here the effective and very effective slices do not agree (the latter is proved by Bachmann \cite{1610.01346}, and the former follows since $\widetilde{s}_0(\unit) \simeq \widetilde{s}_0(\KQ)$). On the other hand, we do not know the relation between the very effective slices of $\KQ$ and the equivariant slices of $\Re(\KQ) \simeq KO_{C_2}$, $C_2$-equivariant real $K$-theory. The very effective slice filtration of $\KQ$ is given by 
  \[
\widetilde{s}_q(\KQ) \simeq \begin{cases}
  \Sigma^{2q,q}\widetilde{s}_0(\unit) & q \equiv 0 \mod (4)\\
  \Sigma^{2q,q} \M\Z/2 & q \equiv 1 \mod (4)\\ 
  \Sigma^{2q,q} \M\Z & q \equiv 2 \mod (4) \\
   0 & q \equiv 3 \mod(4) \\ 
\end{cases}
  \]
  Here the suspensions are such that it is possible that the realization of the $q$-th slice is the $2q$-th equivariant slice of the realization. We save any further investigation for future work. 
\end{rem}
\section{The slice spectral sequence}
Suppose we are in the situation of \Cref{thm:effvshhr}. We have slice spectral sequences in motivic and equivariant homotopy\footnote{All the spectral sequences we consider in this section will converge, for example by \cite[Theorem 8.12]{hoyois_hhm} and \cite[Theorem 4.42]{hhr}.}
\[
E_1^{p,q,w} \cong \pi_{p,w}s_q(E) \implies \pi_{p,w}(E)
\]
and
\[
E_1^{p,q,w} \cong \pi_{p,w}P_{2q}^{2q}(\Re(R)) \implies \pi_{p,w}(\Re(E)),
\]
and moreover Betti realization induces an isomorphism between the two spectral sequences.\footnote{Note that we are implicitly using the even equivariant slice filtration, and ignoring the odd slices, which are all contractible.} It follows that we can use naturality to determine differentials in the motivic spectral sequence from known differentials in equivariant homotopy theory.  
\subsection{The motivic and equivariant slice spectral sequences}
We begin by reminding the reader of the equivariant and motivic cohomology of a point with $\F_2$-coefficients. We start with the motivic case, which is easier to describe. Let $k$ be a field, then there is a canonical element $\tau \in H^{0,1}(\Spec(k),\F_2)$, such that 
\[
\pi_{*,*}(\M \F_2) \cong H^{-*,-*}(\Spec(k),\F_2) \cong k_*^M(k)[\tau]
\]
where $k_*^M(k)$ denotes the mod 2 Milnor $K$-theory of $k$, and $k_n^M(k)$ has cohomological bidegree $(n,n)$ - this is a consequence of \cite{voe_z2}, see \cite[Remark 4.4]{di_adss} for a convenient reference. 

Now assume that $k \subseteq \R$ is a real closed field (that is, $k$ is a subfield of $\mathbb{R}$ that is not algebraically closed, but $k(\sqrt{-1})$ is). In this case, we have an isomorphism $k_*^M(k) \cong \F_2[\rho]$ \cite[Corollary X.6.9]{lam_quadratic_forms}, where $\rho$ represents $[-1] \in k_1^M(k) \cong k^\times/(k^\times)^2$. Together, we conclude that $H^{*,*}(\Spec(k),\F_2) \cong \F_2[\tau,\rho]$, where $\tau$ has homological bidegree $(0,-1)$ and $\rho$ has homological bidegree $(-1,-1)$. We denote this ring by $\mathbb{M}_2$. 

The $C_2$-equivariant Bredon cohomology of a point is given in \cite[Proposition 6.2]{hk_real}. It is useful to first describe some elements in it. There is an element $a_\sigma \in \pi^{C_2}_{-1,-1}S^0$ corresponding to the inclusion $S^{0,0} \to S^{1,\sigma}$. Following standard conventions, we denote the image of $a_{\sigma}$ in $\pi_{\ast,\ast}H\underline{F}_2 \cong H^{-\ast,-\ast}(\mathrm{pt},H\underline{\F}_2)$ as $a_\sigma$ as well. Under realization, the class $\rho \in H^{*,*}(\Spec(k),\F_2)$ maps to the class $a_{\sigma}$. 

The other distinguished class is denoted $u_{\sigma}$, and is the element corresponding to the generator of $H_1^{C_2}(S^{1,\sigma},\underline{\F}_2) \cong \pi^{C_2}_{0,-1}(H\underline{\F}_2)$. Under realization the class $\tau$ maps to the class $u_{\sigma}$. 

With this in mind, the Bredon equivariant cohomology of a point can be described by 
\[
\pi_{*,*}H\underline{\F}_2 \cong H^{-*,-*}(\mathrm{pt},H\underline{\F}_2) \cong \F_2[u_{\sigma},a_{\sigma}] \oplus \F_2 \left\{ \frac{\theta}{u_{\sigma}^i a_{\sigma}^j}\right\}, \quad i,j \ge 0. 
\]
Of course, some care must be taken when interpreting this because $u_{\sigma}$ and $a_{\sigma}$ are not actually invertible. Here the homological bidegrees are given by $|u_{\sigma}|= (-1,0), |a_{\sigma}| = (-1,-1)$ and $|\theta| = (2,0)$. The element $\theta$ is both $u_{\sigma}$ and $a_{\sigma}$-torsion, and is infinitely divisible by the same elements, and any two elements in $\F_2 \left\{ \frac{\theta}{u_{\sigma}^i a_{\sigma}^j}\right\}$ multiply trivially. 
\subsection{Mod $2$ $\BP$ and $\BP\langle n \rangle$}\label{sec:mod2calcs}
Let $k \subseteq \R$ be a real closed field, so that the class $\rho$ is non-zero. We define mod $2$ version of $\BP$ and $\BP\langle n \rangle$ by
\[
\BP/2 = \MGL_{(2)}/(\{2,a_i \mid i \ne p^j-1 \})
\]
and 
\[
\BP \langle n \rangle/2 = \MGL_{(2)}/(\{2,a_i \mid i \ne 2^j -1, 0 \le k \le n\}).
\]
If we write $v_i = a_{2^i-1}$, then the corresponding topological spectral spectra have coefficient rings
\[
(BP/2)_* \cong BP_*/2 \cong \F_2[v_1,v_2,\ldots] \quad \text{ and } \quad (BP\langle n \rangle/2)_* \cong BP\langle n \rangle_*/2 \cong \F_2[v_1,\ldots,v_n]. 
\]
Then, the calculations of Levine and Tripathi \cite[Corollary 4.6]{levine_tripathi} give
\[
s_q(\BP/2) \simeq \Sigma^{2q,q}M\F_2 \otimes BP_{2q}/2 \quad \text{ and } \quad s_q\BP\langle n \rangle \simeq \Sigma^{2q,q}M\F_2 \otimes {BP\langle n \rangle_{2q}}/2.
\]

The associated slice spectral for $\BP/2$ then has the form
\[
E_1^{p,q,w} \cong H^{2q-p,q-w}(\Spec(k),BP_{2q}/2) \implies \pi_{q,w}\BP/2.
\]
When $q=0$, we have $s_0(BP/2) \simeq \M\F_2$, and so there are classes $\tau \in \pi_{0,-1}s_0(\BP/2)$ and $\rho \in \pi_{-1,-1}s_0(\BP/2)$. In particular, in the slice spectral sequence for $\BP/2$ we have $\tau^m \in E_1^{0,0,-m}$ and $\rho^m \in E_1^{m,m,-m}$. Moreover, the classes $v_k$ give rise to elements in $\pi_{2k,k}s_k(\BP/2) \cong \pi_0\M\F_2$. 
\begin{prop}
  The $E_1$-term for the slice spectral sequence for $\BP/2$ is given by
  \[
E_1^{p,q,w} \cong \F_2[\rho,\tau,v_1,v_2,\ldots]
  \]
  with tri-degrees $|\tau| = (0,0,-1)$, $|\rho| = (-1,0,-1)$, and $|v_i| = (2^{i+1}-2,2^i-1,2^i-1)$. The differentials $d_i(\tau^{2^{k}})$ are zero for $i < 2^{k}-1$, and 
  \[
d_{2^{k}-1}(\tau^{2^{k}}) = v_k \rho^{2^{k+1}-1}.
  \]
\end{prop}
\begin{proof}
  That the $E_1$-term has the claimed form is clear from the discussion above. To determine the differentials, we use naturality to compare with the differentials for the $C_2$-equivariant slice spectral sequence for $BP\R/2$. By \Cref{thm:comparasionthmchromatic} we can compute the slices for $BP\R/2$ (or this can also be deduced using the method of twisted monoid rings in \cite{hhr}). In particular, we deduce that the odd slices for $BP\R/2$ are trivial, and the even slices are given by wedges of $H\underline{\mathbb{F}}_2$, where the wedges are indexed by monomials in $BP_*/2$. We chose to actually work with the even slice filtration, as this Betti realization induces an isomorphism on slices.  

   It follows from the discussion in the previous section that in positive degrees the $E_1$-term of the slice spectral sequence has the same form as the motivic slice spectral sequence, namely $\F_2[a_{\sigma},u_{\sigma},\overline{v_1},\overline{v_2},\ldots]$, where $\overline{v}_i \in \pi_{2(2^i-1),2^i-1}BP\R$ is the equivariant lift of $v_i \in \pi_{2(2^i-1)}BP$. In particular, there is an injection on the $E_1$-page of the corresponding slice spectral sequences. 

  The differentials in the equivariant slice spectral sequence for $BP\R/2$ can be determined by the differentials in the slice spectral sequence for $BP\R$ itself. The differentials in the latter can be determined by the work of Hu and Kriz \cite{hk_real} or as a consequence of Hill, Hopkins, Ravenel \cite[Theorem 9.9]{hhr}, and are fully described in \cite[Proposition 3.4]{hurewicz}. To wit, the $E_1$-term is given in positive degrees by $\Z[a_{\sigma},u_{2\sigma},\overline{v}_i]$. Here $u_{2\sigma} \in \pi_{-2,-2}H\underline{\Z}$ is a lift of the class $u_{\sigma}^2 \in \pi_{-2,-2}H\underline{\F}_2$ (the class $u_{\sigma}$ does not exist in $\pi_{*,*}H\underline{\Z}$.). The differentials $d_i(\tau^{2^{k}})$ are zero for $i < 2^{k}-1$, and 
  \[
d_{2^{k}-1}(u_{2\sigma}^{2^{k-1}}) = \overline{v}_k a_{\sigma}^{2^{k+1}-1}.
\]
Note that we use a different grading convention to Hill, Hopkins, and Ravenel - we choose to start our spectral sequence on the $E_1$-page instead of the $E_2$-page. Moreover, we work with the even slice filtration. This has the effect that the first differential in our spectral sequence is a $d_1$, which corresponds to a $d_3$-differential for Hill, Hopkins, and Ravenel. Similarly, their $d_{2^k}$-differential corresponds to a $d_{2^k-1}$-differential in our grading.

By considering the map of slice spectral sequences induced by the quotient $BP\R \to BP\R/2$, we deduce that the equivariant slice spectral sequence for $BP\R/2$ has differentials 
\[
d_{2^k-1}(u_{\sigma}^{2^k}) = \overline{v}_ka_{\sigma}^{2^{k+1}-1}. 
\]
This can also be determined using the methods of Hu and Kriz, cf \cite[Section 3]{hk_real}. 

From the calculations of the $E_1$-pages, we see that the map of spectral sequences from $\BP/2$ to $BP\R/2$ is an injection on the $E_1$-page. It follows that the equivariant differential $d_1(\tau) = \rho$ implies that there is a motivic differential $d_1(u_{\sigma}) = a_{\sigma}$, and that the spectral sequence is an injection on the $E_2$-page. Inductively, we deduce that the morphism of spectral sequences is an injection on the $E_k$-page, and that the motivic differentials are as claimed. 
\end{proof}
This $E_1$-term is remarkably similar to the $\rho$-Bockstein spectral sequence for computing $\Ext_{E}(\mathbb{M}_2,\mathbb{M}_2)$, as studied by Hill in \cite{hill_motivic}, where $E$ is the sub-algebra of the mod 2 motivic Steenrod algebra generated by the Milnor primitives. Using a similar analysis to Hill (see \cite[Corollary 3.3]{hill_motivic}), and the previous proposition, one arrives at the following result, which should be compared to Yagita's computation \cite[Theorem 6.5]{yagita}.
\begin{thm}
  Over a real closed field $k \subseteq \R$, the $E_\infty$-term of the slice spectral sequence $\BP/2$ is given additively by 
  \[
\F_2[\rho,\tau,v_i(j) \mid i > 0, j \ge 0]
  \]
subject to the relations
  \[
  \begin{split}
    \tau^2 &= 0 \\
    \rho^{2^{i+1}-1}v_i(j) &= 0  \\
    v_i(j) \cdot v_k(\ell) &= v_i(j+2^{k-i}\ell)\cdot v_k(0) \text{ when }k \ge i.
  \end{split}
\]
  Here $v_i(j)$ is represented on the $E_1$-page of the slice spectral sequence by $\tau^{2^{i+1}j}v_i$ (so in particular, $v_i(0)$ is represented by $v_i$). 
\end{thm}
A similar analysis gives the following for the truncated Brown--Peterson spectra.  
\begin{thm}
  Over a real closed field $k \subseteq \R$, the $E_1$-term of the slice spectral sequence for $\BP\langle n \rangle /2$ is given by 
  \[
E_1^{p,q,w} \cong \F_2[\rho,\tau,v_1,\ldots,v_n]. 
  \]
The differentials $d_i(\tau^{2^{k}})$ are zero for $i < 2^{k}-1$, and 
  \[
d_{2^{k}-1}(\tau^{2^{k}}) = v_k \rho^{2^{k+1}-1}
  \]
  for $1 \le k \le n$. We deduce that the $E_\infty$-term of the slice spectral sequence of $\BP\langle n \rangle /2$ is given additively by 
  \[
\F_2[\rho,\tau,t_{n+1}, v_i(j) \mid 0 < i \le n, j \ge 0]
  \]
subject to the relations
\[
\begin{split}
    \tau^2 &= 0 \\
    \rho^{2^{i+1}-1}v_i(j) &= 0  \\
      v_i(j) \cdot v_k(\ell) &= v_i(j+2^{k-i}\ell)\cdot v_k(0) \text{ when }k \ge i\\
      v_i(j) &= t_{n+1} v_i(j-2^{n-i}) \text{ when } j \ge 2^{n-i}. 
\end{split}
\]
  Here $v_i(j)$ is represented on the $E_1$-page of the slice spectral sequence by $\tau^{2^{i+1}j}v_i$, while $t_{n+1}$ is represented on the $E_1$-page by $\tau^{2^{n+1}}$. 
\end{thm}
There can be non-trivial extensions in these spectral sequences. For example, when $n=1$, the $E_\infty$-page of the slice spectral sequence for $\BP\langle 1 \rangle/2$ is given by $\F_2[\rho,\tau,t_2,v_1]/(\tau^2,\rho^3 v_1)$. In particular, in weight zero it is given by
\[
\F_2[t_2v_1^4]\{ 1,\rho v_1,(\rho^2 v_1^2,\tau v_1),\tau \rho v_1^2,\tau \rho^2 v_1^3,0,0,0\}.
\]
On the other hand, we have that $\BP \langle 1 \rangle/2 \simeq \kgl/2$, so that in weight 0 this computes the mod 2 algebraic $K$-theory of $k$. By Suslin \cite{suslin_ktheory}, over $\Spec(\R)$ we have
\[
K_n(\R,\Z/2) \cong \begin{cases}
  \Z/2 &\text{ when } n \equiv 0,1,3,4 \mod (8), \\
  \Z/4 & \text { when } n \equiv 2 \mod (8), \\
  0 & \text{ otherwise.}
\end{cases}
\]
If follows that there must be a non-trivial additive extension between $\rho^2v_1^2$ and $\tau v_1$, and their $t_2v_1^4$-multiples.

\bibliography{bib_motivic}\bibliographystyle{gtart}

\begin{thebibliography}{}
\providecommand\bibmarginpar{\leavevmode\marginpar}
\def\urlstyle#1{{\tt #1}}

\bibitem{1712.01349}
\textbf{A Ananyevskiy}, \textbf{O R\"ondigs}, \textbf{P\,A {\O}stv{\ae}r},
  \emph{On very effective hermitian {$K$}-theory}  (2017) arXiv:1712.01349

\bibitem{1610.01346}
\textbf{T Bachmann}, \href{http://dx.doi.org/10.1112/topo.12032} {\emph{The
  generalized slices of {H}ermitian {K}-theory}}, {J}ournal of {T}opology 10
  (2017) 1124--1144

\bibitem{di_cell}
\textbf{D Dugger}, \textbf{D\,C Isaksen},
  \href{https://doi.org/10.2140/agt.2005.5.615} {\emph{Motivic cell
  structures}}, Algebr. Geom. Topol. 5 (2005) 615--652

\bibitem{di_adss}
\textbf{D Dugger}, \textbf{D\,C Isaksen},
  \href{https://doi.org/10.2140/gt.2010.14.967} {\emph{The motivic {A}dams
  spectral sequence}}, Geom. Topol. 14 (2010) 967--1014

\bibitem{greenlees_meier}
\textbf{J\,P\,C Greenlees}, \textbf{L Meier},
  \href{http://dx.doi.org/10.2140/agt.2017.17.3547} {\emph{Gorenstein duality
  for real spectra}}, Algebr. Geom. Topol. 17 (2017) 3547--3619

\bibitem{grso_slices}
\textbf{J\,J Guti\'errez}, \textbf{O R\"ondigs}, \textbf{M Spitzweck},
  \textbf{P\,A {\O}stv{\ae}r}, \href{http://dx.doi.org/10.1112/jtopol/jts015}
  {\emph{Motivic slices and coloured operads}}, J. Topol. 5 (2012) 727--755

\bibitem{heller_ormsby}
\textbf{J Heller}, \textbf{K Ormsby}, \href{https://doi.org/10.1090/tran6647}
  {\emph{Galois equivariance and stable motivic homotopy theory}}, Trans. Amer.
  Math. Soc. 368 (2016) 8047--8077

\bibitem{hill_motivic}
\textbf{M\,A Hill}, \href{http://dx.doi.org/10.1016/j.jpaa.2010.06.017}
  {\emph{Ext and the motivic {S}teenrod algebra over {$\mathbf R$}}}, J. Pure
  Appl. Algebra 215 (2011) 715--727

\bibitem{hill_primer}
\textbf{M\,A Hill}, \href{https://doi.org/10.4310/HHA.2012.v14.n2.a9}
  {\emph{The equivariant slice filtration: a primer}}, Homology Homotopy Appl.
  14 (2012) 143--166

\bibitem{hhr}
\textbf{M\,A Hill}, \textbf{M\,J Hopkins}, \textbf{D\,C Ravenel},
  \href{http://dx.doi.org/10.4007/annals.2016.184.1.1} {\emph{On the
  nonexistence of elements of {K}ervaire invariant one}}, Ann. of Math. (2) 184
  (2016) 1--262

\bibitem{hill_meier}
\textbf{M\,A Hill}, \textbf{L Meier},
  \href{http://dx.doi.org/10.2140/agt.2017.17.1953} {\emph{The {$C_2$}-spectrum
  {${\rm Tmf}_1(3)$} and its invertible modules}}, Algebr. Geom. Topol. 17
  (2017) 1953--2011

\bibitem{hoyois_hhm}
\textbf{M Hoyois}, \href{http://dx.doi.org/10.1515/crelle-2013-0038}
  {\emph{From algebraic cobordism to motivic cohomology}}, J. Reine Angew.
  Math. 702 (2015) 173--226

\bibitem{hk_real}
\textbf{P Hu}, \textbf{I Kriz},
  \href{https://doi.org/10.1016/S0040-9383(99)00065-8} {\emph{Real-oriented
  homotopy theory and an analogue of the {A}dams-{N}ovikov spectral sequence}},
  Topology 40 (2001) 317--399

\bibitem{isaksen_shkembi}
\textbf{D\,C Isaksen}, \textbf{A Shkembi},
  \href{https://doi.org/10.1017/is011004009jkt154} {\emph{Motivic connective
  {$K$}-theories and the cohomology of {A}(1)}}, J. K-Theory 7 (2011) 619--661

\bibitem{kelly_thesis}
\textbf{S Kelly}, \emph{Triangulated categories of motives in positive
  characteristic}  (2013) arXiv:1305.5349

\bibitem{lam_quadratic_forms}
\textbf{T\,Y Lam}, \emph{Introduction to quadratic forms over fields},
  volume~67 of \emph{Graduate Studies in Mathematics}, American Mathematical
  Society, Providence, RI (2005)

\bibitem{levine_tower}
\textbf{M Levine}, \href{https://doi.org/10.1112/jtopol/jtm004} {\emph{The
  homotopy coniveau tower}}, J. Topol. 1 (2008) 217--267

\bibitem{levine_tripathi}
\textbf{M Levine}, \textbf{G\,S Tripathi}, \emph{Quotients of {MGL}, their
  slices and their geometric parts}, Doc. Math.  (2015) 407--442

\bibitem{hurewicz}
\textbf{G Li}, \textbf{X\,D Shi}, \textbf{G Wang}, \textbf{Z Xu},
  \href{http://dx.doi.org/10.1016/j.aim.2018.11.002} {\emph{Hurewicz images of
  real bordism theory and real {J}ohnson-{W}ilson theories}}, Adv. Math. 342
  (2019) 67--115

\bibitem{lurie_htt}
\textbf{J Lurie}, \href{http://dx.doi.org/10.1515/9781400830558} {\emph{Higher
  topos theory}}, volume 170 of \emph{Annals of Mathematics Studies}, Princeton
  University Press, Princeton, NJ (2009)

\bibitem{ha}
\textbf{J Lurie}, \emph{Higher {A}lgebra} (2017)Draft available from author's
  website as \url{http://www.math.harvard.edu/~lurie/papers/HA.pdf}

\bibitem{mandell_may}
\textbf{M\,A Mandell}, \textbf{J\,P May},
  \href{https://doi.org/10.1090/memo/0755} {\emph{Equivariant orthogonal
  spectra and {$S$}-modules}}, Mem. Amer. Math. Soc. 159 (2002) x+108

\bibitem{mnn_15}
\textbf{A Mathew}, \textbf{N Naumann}, \textbf{J Noel},
  \href{https://doi.org/10.1016/j.aim.2016.09.027} {\emph{Nilpotence and
  descent in equivariant stable homotopy theory}}, Adv. Math. 305 (2017)
  994--1084

\bibitem{mazel_gee}
\textbf{A Mazel-Gee}, \href{http://nyjm.albany.edu:8000/j/2016/22_57.html}
  {\emph{Quillen adjunctions induce adjunctions of quasicategories}}, New York
  J. Math. 22 (2016) 57--93

\bibitem{morel_connectivity}
\textbf{F Morel}, \href{https://doi.org/10.1007/978-94-007-0948-5_7} {\emph{On
  the motivic {$\pi_0$} of the sphere spectrum}}, from ``Axiomatic, enriched
  and motivic homotopy theory'', NATO Sci. Ser. II Math. Phys. Chem. 131,
  Kluwer Acad. Publ., Dordrecht (2004)  219--260

\bibitem{morel_voe_a1}
\textbf{F Morel}, \textbf{V Voevodsky},
  \href{http://www.numdam.org/item?id=PMIHES_1999__90__45_0} {\emph{{${\bf
  A}^1$}-homotopy theory of schemes}}, Inst. Hautes \'Etudes Sci. Publ. Math.
  (1999) 45--143 (2001)

\bibitem{nso_landweber}
\textbf{N Naumann}, \textbf{M Spitzweck}, \textbf{P\,A {\O}stv{\ae}r},
  \emph{Motivic {L}andweber exactness}, Doc. Math. 14 (2009) 551--593

\bibitem{mbp_rationals}
\textbf{K\,M Ormsby}, \textbf{P\,A {\O}stv{\ae}r},
  \href{https://doi.org/10.2140/gt.2013.17.1671} {\emph{Motivic
  {B}rown-{P}eterson invariants of the rationals}}, Geom. Topol. 17 (2013)
  1671--1706

\bibitem{pelaez_functor}
\textbf{P Pelaez}, \href{http://dx.doi.org/10.1017/is013001013jkt196} {\emph{On
  the functoriality of the slice filtration}}, J. K-Theory 11 (2013) 55--71

\bibitem{robalo}
\textbf{M Robalo}, \href{https://doi.org/10.1016/j.aim.2014.10.011}
  {\emph{{$K$}-theory and the bridge from motives to noncommutative motives}},
  Adv. Math. 269 (2015) 399--550

\bibitem{ro_slices_kq}
\textbf{O R{\"o}ndigs}, \textbf{P\,A {\O}stv{\ae}r},
  \href{https://doi.org/10.2140/gt.2016.20.1157} {\emph{Slices of hermitian
  {$K$}-theory and {M}ilnor's conjecture on quadratic forms}}, Geom. Topol. 20
  (2016) 1157--1212

\bibitem{cell_KQ}
\textbf{O R{\"o}ndigs}, \textbf{M Spitzweck}, \textbf{P\,A {\O}stv{\ae}r},
  \emph{Cellularity of hermitian {K}-theory and {W}itt theory}
  (2016) arXiv:1603.05139

\bibitem{rso_slices}
\textbf{O R\"{o}ndigs}, \textbf{M Spitzweck}, \textbf{P\,A \O~stv\ae r},
  \href{http://dx.doi.org/10.4007/annals.2019.189.1.1} {\emph{The first stable
  homotopy groups of motivic spheres}}, Ann. of Math. (2) 189 (2019) 1--74

\bibitem{spitzweck_slices}
\textbf{M Spitzweck}, \href{http://projecteuclid.org/euclid.hha/1296223886}
  {\emph{Relations between slices and quotients of the algebraic cobordism
  spectrum}}, Homology Homotopy Appl. 12 (2010) 335--351

\bibitem{spitzweck_landweber}
\textbf{M Spitzweck}, \href{https://doi.org/10.1017/is010008019jkt128}
  {\emph{Slices of motivic {L}andweber spectra}}, J. K-Theory 9 (2012) 103--117

\bibitem{1207.4078}
\textbf{M Spitzweck}, \emph{A commutative {$\mathbf P^1$}-spectrum representing
  motivic cohomology over {D}edekind domains}, M\'{e}m. Soc. Math. Fr. (N.S.)
  (2018) 110

\bibitem{mot_twisted}
\textbf{M Spitzweck}, \textbf{P\,A {\O}stv{\ae}r},
  \href{http://dx.doi.org/10.2140/agt.2012.12.565} {\emph{Motivic twisted
  {$K$}-theory}}, Algebr. Geom. Topol. 12 (2012) 565--599

\bibitem{suslin_ktheory}
\textbf{A Suslin}, \href{https://doi.org/10.1007/BF01394024} {\emph{On the
  {$K$}-theory of algebraically closed fields}}, Invent. Math. 73 (1983)
  241--245

\bibitem{ullman_thesis}
\textbf{J\,R Ullman},
  \href{http://gateway.proquest.com/openurl?url_ver=Z39.88-2004&rft_val_fmt=info:ofi/fmt:kev:mtx:dissertation&res_dat=xri:pqm&rft_dat=xri:pqdiss:0829532}
  {\emph{On the {R}egular {S}lice {S}pectral {S}equence}}, ProQuest LLC, Ann
  Arbor, MI (2013)Thesis (Ph.D.)--Massachusetts Institute of Technology

\bibitem{ullman_agt}
\textbf{J Ullman}, \href{http://dx.doi.org/10.2140/agt.2013.13.1743} {\emph{On
  the slice spectral sequence}}, Algebr. Geom. Topol. 13 (2013) 1743--1755

\bibitem{voe_a1}
\textbf{V Voevodsky}, \emph{{$\mathbf A^1$}-homotopy theory}, from
  ``Proceedings of the {I}nternational {C}ongress of {M}athematicians, {V}ol.
  {I} ({B}erlin, 1998)'', Extra Vol. I (1998)  579--604

\bibitem{voe_open}
\textbf{V Voevodsky}, \emph{Open problems in the motivic stable homotopy
  theory. {I}}, from ``Motives, polylogarithms and {H}odge theory, {P}art {I}
  ({I}rvine, {CA}, 1998)'', Int. Press Lect. Ser. 3, Int. Press, Somerville, MA
  (2002)  3--34

\bibitem{voe_z2}
\textbf{V Voevodsky}, \href{https://doi.org/10.1007/s10240-003-0010-6}
  {\emph{Motivic cohomology with {${\bf Z}/2$}-coefficients}}, Publ. Math.
  Inst. Hautes \'Etudes Sci.  (2003) 59--104

\bibitem{yagita}
\textbf{N Yagita}, \href{http://dx.doi.org/10.1112/S0024611504015084}
  {\emph{Applications of {A}tiyah-{H}irzebruch spectral sequences for motivic
  cobordism}}, Proc. London Math. Soc. (3) 90 (2005) 783--816

\bibitem{zahler_anss}
\textbf{R Zahler}, \href{http://dx.doi.org/10.2307/1970821} {\emph{The
  {A}dams-{N}ovikov spectral sequence for the spheres}}, Ann. of Math. (2) 96
  (1972) 480--504

\end{thebibliography}

\end{document}